\documentclass{amsart}
\usepackage[foot]{amsaddr}
\usepackage[english]{babel}
\usepackage[utf8]{inputenc}
\usepackage{listings} 
\usepackage{graphicx}
\usepackage{dsfont}
\usepackage{amsmath,amssymb,amsthm}
\usepackage{color}   
\usepackage{cancel} 
  
\theoremstyle{plain}

\newtheorem{thm}{Theorem}
\newtheorem{lem}[thm]{Lemma}
\newtheorem{teo}[thm]{Theorem}
\newtheorem{cor}[thm]{Corollary}
\newtheorem{prop}[thm]{Proposition}

\theoremstyle{definition}
\newtheorem{defi_aux}[thm]{Definition}
\newenvironment{defi}{ \begin{defi_aux}}{%
    \end{defi_aux}\bigskip
}
\newtheorem{ex_aux}[thm]{Example}
\newenvironment{ex}{ \begin{ex_aux}}{%
    \end{ex_aux}\bigskip
}

\theoremstyle{remark}
\newtheorem{rmk_aux}[thm]{Remark}
\newenvironment{rmk}{\bigskip\begin{rmk_aux}}{%
    \end{rmk_aux}\bigskip
}

\newcommand{\R}{\mathds{R}}
\newcommand{\N}{\mathds{N}}
\newcommand{\Z}{\mathds{Z}}
\newcommand{\C}{\mathds{C}}
\newcommand{\liZ}{l_\infty(\Z)}
\newcommand{\espan}[1]{\text{span}\left\{ #1 \right\}}

\newcommand{\lra}{\longrightarrow}
\newcommand{\iinZ}{{i\in\Z}}
\newcommand{\cF}{\mathcal{F}}
\newcommand{\cC}{\mathcal{C}}

\newcommand{\kZ}{{2^{-k}\Z}}

\newcommand{\iZ}{{i\in\Z}}
\newcommand{\Sinf}{S^\infty}
\newcommand{\ze}[2]{\mathcal{Z}\left(#1,#2\right)}
\newcommand{\zek}[3]{\mathcal{Z}\left(#2^{#1},#3\right)}
\newcommand{\norma}[1]{\left\|#1\right\|_\infty}
\newcommand{\de}{\delta_\eta}
\newcommand{\uno}{\mathds{1}}

\begin{document}


\title[Nonlinear stationary subdivision schemes that reproduce trigonometric functions]{Nonlinear stationary subdivision schemes\\that reproduce trigonometric functions}

\author{Rosa Donat} 
\email{donat@uv.es}
\author{Sergio López-Ure\~na}
\email{sergio.lopez-urena@uv.es}

\date{\today}
\address{Departament de Matem\`atiques.  Universitat de Val\`encia. Doctor Moliner Street 50, 46100 Burjassot, Valencia (Spain).}

\begin{abstract}
In this paper we define a family of nonlinear, stationary,
interpolatory  subdivision schemes
with the capability of reproducing conic shapes including polynomials upto
second order. 
Linear, non-stationary, subdivision schemes do also achieve this goal,
but different conic sections require different refinement rules to guarantee
exact reproduction.
On the other hand,  with our construction, exact reproduction of
different conic shapes can be 
achieved using exactly the same nonlinear scheme.
Convergence, stability, approximation and shape preservation
properties of the new schemes are analyzed. In addition, the conditions to obtain $\cC^1$ limit functions are also studied.


\end{abstract}

\maketitle

\smallskip
\noindent \textbf{Keywords.}
nonlinear subdivision schemes, exponential polynomials, reproduction,
monotonicity preservation, approximation, smoothness. 

\smallskip
\noindent \textbf{MSC.} 41A25, 41A30, 65D15, 65D17.


\section{Introduction}

Subdivision refinement is a powerful technique for the design and
representation of curves and surfaces. Subdivision algorithms are
simple to implement and extremely well suited for computer
applications, which make them  very attractive to users interested in
geometric modeling and Computer Aided Geometric Design (CAGD). 

As in \cite{DL95,Dyn92,DLL03}, in this paper we consider that a
subdivision scheme $S=(S^k)_{k=0}^\infty$ is a sequence of
operators $S^k:\liZ\lra\liZ$\footnote{$\liZ:=\{(f_i)_\iZ \,: \,f_i\in\R, \, \exists M>0, |f_i|\leq
M\}$ and $\norma{f}:=\sup_{i\in\Z}|f_i|$.} such that for
 each $f^0\in\liZ$, a sequence $(f^k)_{k=0}^\infty\subset\liZ$
 is
 recursively defined as follows:
$$f^{k+1}:=S^{k}f^{k}\in\liZ, \quad k\geq 0.$$

 
We shall restrict our attention to {\em binary}
subdivision
schemes which are \emph{uniform} and \emph{local}, i.e.
there exists $q\geq 0$ such that for any $f \in \liZ$
\begin{equation} \label{eq:psi}
(S^k f)_{2i+j}=\Psi_j^k(f_{i-q},\ldots,f_{i+q}), \qquad j=0,1, \quad \iZ,
\end{equation}
for some $\Psi_j^k:\R^{2q+1}\lra\R$.
  If the functions $\Psi_j^k$ are linear, the level dependent rules
  can be expressed as follows
\begin{equation} \label{eq:Sk-lin}
\Psi_j^k(f_{i-q},\ldots,f_{i+q})=
\sum_{l=-q}^q a_{j-2l}^k f_{i+l} \qquad \iZ,
\end{equation}
where the sequence $(a^k_i)_\iinZ$ (that has finite support) is the
\emph{mask} of the linear operator $S^k$, and the scheme is called \emph{linear}. If the rules are the
same at all refinement levels $k\geq 0$, then the scheme is said to be
\emph{stationary}, and we simply denote $S^k=S$, $\Psi^k_j = \Psi_j$ and $a^k_i=a_i$, $\forall k \geq 0$.   The data  generated by
binary subdivision schemes are usually associated to the underlying
grids $\kZ$, $k \geq 0$.
 
Subdivision schemes of interest in practical applications need to be
{\em convergent}.
In this paper, we are concerned with the  classical notion of
uniform convergence, which is  relevant in geometric modeling.
\begin{defi} \label{defi:convergent}
A binary subdivision scheme, $S$, is \emph{uniformly convergent} if
\begin{equation} \label{eq:SS_conv}
\forall f^0\in\liZ \quad \exists \Sinf f^0\in\cC(\R) \, : \,
\lim_{k\rightarrow\infty}  \sup_\iZ |f^k_i - (\Sinf f^0)(i2^{-k})|=0.
\end{equation}
\end{defi}
If $S$ is a convergent subdivision scheme,  $\Sinf:\liZ\rightarrow\cC(\R)$ denotes
the operator that
sends any initial data $f^0$ to the continuous limit function
obtained by the subdivision process specified by $S$.  A uniformly
convergent subdivision scheme is $\cC^m$, or $\cC^m$-convergent, if
for any initial data the limit function has continuous derivatives up
to order $m$. 

{\em Interpolatory} subdivision refinement is often used in practical applications. Binary interpolatory subdivision schemes satisfy
\begin{equation} \label{eq:SS_interp}
 f^{k+1}_{2i} = (S^kf^k)_{2i} =f^k_i, \qquad \forall\iZ.
\end{equation}
Hence, they are defined by specifying only the so-called {\em
  insertion rules}, $\Psi^k_1$, that define $f^{k+1}_{2i+1}$ in \eqref{eq:psi}. If $S$ is
interpolatory and convergent then 
\begin{equation} \label{eq:interpolate_initial}
(\Sinf f^0)(i2^{-k})=f_i^k,
\end{equation}
that is, $\Sinf f^0$ interpolates the  data, $f^k$, at each resolution level.
This property may be very 
useful  in  CAGD, where
interpolatory subdivision 
schemes are often used to obtain  specific shapes from an initial
(coarse) set of samples. 

One important property of linear non-stationary subdivision schemes, which
linear stationary subdivision schemes do not have,
is that they can be used to efficiently generate {\em
  conical shapes} (circles, ellipses...) from
a coarse  initial set of samples.
For example,  an initial  representation of
$n$ equidistant  points on
the unit  circle is obtained by considering  $f^0
= (G(i))_\iinZ$  where $G(t)  = (\cos(\gamma t), \sin(\gamma t))$,
with $\gamma=2 \pi/n$. 
 From this initial coarse representation, the circle can be
obtained  by using interpolatory
subdivision schemes  capable of  
{\em reproducing} the trigonometric functions 
 The relevant definition on function reproduction is provided below
 for completeness. 
\begin{defi}
\label{defi:rep_gen}
A convergent subdivision scheme $S$ \emph{reproduces} the set of continuous functions $\mathcal{V}$ if
\begin{equation} \label{def:F-rep}
 f^0_i = F(i) \quad \forall  \iinZ \quad \Longrightarrow \quad \Sinf f^0 = F, \qquad \forall
 F\in\mathcal{V}.
\end{equation}
\end{defi}

This paper is concerned with the efficient reproduction of conical sections by means of interpolatory
stationary {\em 
  nonlinear} subdivision schemes.
 In particular,  we consider 
 the (finite dimensional) spaces 
\begin{equation} \label{eq:Vngamma}
W_{0,\gamma}=
\espan{1,  \exp(\gamma t), \exp(-\gamma t)}, \qquad  0 \neq \gamma \in \mathbb{R}\cup \imath (-\pi,\pi),
\qquad \imath = \sqrt{-1}, 
\end{equation}
 which  can be used to represent circles, ellipses and hyperbolic
 functions centered anywhere in the plane. {

As mentioned before, any linear scheme 
reproducing   functions in this space 
 must  be necessarily  non-stationary.  In addition,  the value of
$\gamma$  must be known in order to determine the level-dependent rules of the non-stationary scheme  (see e.g. \cite{CR11,DLL03} and section
\ref{sec:definition}). In practice, this value is estimated from the samples 
\cite{DLL03}, but the dependence of these  subdivision  
schemes on the value of the parameter $\gamma$ 
may be seen as a 
drawback when we desire to obtain shapes
composed of different conical sections. In
this case,  several values of the 
parameter  $\gamma$  might have to be estimated from the initial
samples, and different schemes would have to be used for exact
reproduction of the different sections.
There is nowadays a 
relatively large  body of research concerning linear, non-stationary
subdivision   schemes and their
properties (see e.g. \cite {CR11,DL95,DLL03} and references therein). 

In this paper, we shall construct 
and analyze a family of {\em interpolatory, nonlinear,
stationary} subdivision
schemes with the
capability to reproduce\footnote{$\Pi_n$
  is the space of polynomials of degree $n$} $\Pi_{2} \cup W_{0,\gamma}$ for all $\gamma
\in \R \cup \imath[-\frac34\pi,\frac34\pi]$, independently of the
value of  $\gamma$, under rather general conditions.  Hence, our schemes
 provide  a simple and
effective way to obtain  curves composed of different
conic/hyperbolic sections by recursive subdivision, obtaining exact reproduction in each of the conic sections.


The derivation of the nonlinear schemes proposed in this paper is
based on  the {\em   orthogonal rules} that  
annihilate the space $W_{0,\gamma}$. Orthogonal rules
are 
one possible
way to derive non-stationary subdivision schemes capable of
reproducing spaces of {\em exponential polynomials} (see \cite{DLL03}
or section \ref{sec:expo}). 

The paper is organized as follows: In sections \ref{sec:preliminaries}
and  \ref{sec:expo} we recall the basic results about convergence of
stationary subdivision schemes, both linear and nonlinear, as well as
the essential ingredients on the spaces of exponential polynomials and
the derivation of the orthogonal rules that are relevant in our
construction.  Sections \ref{sec:definition} and \ref{sec:convergence}
are dedicated to the definition and convergence analysis of the proposed
nonlinear  subdivision schemes. The preservation of monotonicity is
analyzed  in section \ref{sec:mono}. In section
\ref{sec:smoothness} we perform a study of the regularity of the limit
functions obtained from monotone initial data,
following (loosely) some ideas in \cite{Kuijt98}. The approximation
capabilities and the stability are treated in section
\ref{sec:approx}. Section \ref{sec:numeric} shows some numerical
experiments that fully  support our theoretical results. We close in
section \ref{sec:conclusions} with some conclusions ans future perspectives.
 
\section{ Convergence of stationary subdivision schemes.
\label{sec:preliminaries}}  

In what follows, we  set the
notation for the remainder of the paper and briefly
review those tools used in the analysis of stationary subdivision schemes (linear or not) that are relevant in our development. The
interested reader is referred to \cite{CMD91,Dyn92}
for a more complete description of the  theory of {\em linear} subdivision schemes,
and to \cite{ADLT06,ADS16,DRS04,DL-US17,HO10}
for the relevant theory of nonlinear subdivision schemes that can be written as a nonlinear
perturbation of a convergent linear scheme.
  In the following we
 use the letter $T$ only for linear schemes, while $S$ shall denote a general subdivision scheme (linear or not).

 A well-known necessary
condition for the convergence of a stationary subdivision scheme is that
\begin{equation} \label{eq:sum+sum}
 \sum_\iZ a_{2i} = \sum_\iZ a_{2i+1} = 1.
\end{equation}
The last relation is equivalent to the {\em reproduction of constant
  sequences} property and  implies the existence of the  {\em first difference scheme
  $T^{[1]}$}, characterized by the following property
\begin{equation*} 
 \nabla T f = T^{[1]} \nabla f ,\qquad \forall f \, \in \liZ,
\end{equation*}
where $\nabla f = (f_{i+1}-f_i)_\iZ$. 
Moreover, if $T$ is a linear subdivision scheme satisfying \eqref{eq:sum+sum}, then $T$ is convergent if, and only if,
\begin{equation} \label{eq:Slin-Lmu}
 \exists L>0, \eta<1 : \quad \norma{(T^{[1]} \circ T^{[1]} \circ
\overset{(L)}{\cdots} \circ T^{[1]}) f}\leq \eta \norma{f} \quad \forall f
\in \liZ.
\end{equation}

%
In \cite{ADL11,ADLT06,ADS16,DL08,DRS04,DL-US17,HO10},
the authors 
construct and analyze several nonlinear subdivision
schemes that can be described as a (rather specific) nonlinear
perturbation of a convergent linear scheme.  In this paper, we are
interested in 
subdivision schemes of the form 
\begin{equation}  \label{eq:nonlinearSS}
 S f=T f+\cF(  \nabla f),
  \qquad \forall f \in \liZ,
\end{equation}
where $\cF:\liZ\rightarrow
\liZ$ may be a nonlinear operator and $T$ is a linear and
convergent subdivision scheme (so $T^{[1]}$ exists). It is easy to see
that  schemes of the form specified in \eqref{eq:nonlinearSS}  always admit a first-difference
 scheme. Indeed, applying the difference operator, $\nabla$ to
 \eqref{eq:nonlinearSS} we get
\[ \nabla S f= \nabla T f + \nabla \cF (\nabla f)=T^{[1]} \nabla f +
\nabla \cF(\nabla f). \]
 Hence, $S^{[1]}$ is defined as
\begin{equation} \label{eq:nonlinearS1}
 S^{[1]} f= T^{[1]}  f +  \nabla \cF(f), \qquad  \forall f \in \liZ . 
\end{equation}

Then, convergence
can be proven using the following 
 result from \cite{ADS16} 
(notice the similarity with \eqref{eq:Slin-Lmu}).
\begin{teo}
 \label{teo:conv1} 
Let $S$ be an interpolatory subdivision scheme of the form
\eqref{eq:nonlinearSS}. 
If
\begin{align}
\text{\textbf{C1.}} & \qquad \exists M>0: &  ||\cF(f) ||_{\infty} \leq M\norma{f}, &&   \forall f \in \liZ,  \label{eq:teo:conv1}\\
\text{\textbf{C2.}} & \qquad \exists L>0, \, \, 0<\eta<1 : & \norma{(
  S^{[1]}\circ S^{[1]}\circ 
\overset{L}{\cdots} \circ S^{[1]})(f)}\leq \eta \norma{f}, && \forall f \in
\liZ, \nonumber 
\end{align}
then $S$ is  convergent. If  $T$ is $\cC^\alpha$, then $S$ is $\cC^{\beta}$ with
$\beta=\min \{\alpha, -\log_2(\eta)/L \}$.  
\end{teo} 

\begin{rmk}
 We write $F\in \cC^{\alpha}$, being $\alpha = n + \kappa$, with $n\in\N$ and $0<\kappa<1$, if $F^{(n)}$ exists and satisfies
$$|F^{(n)}(x)-F^{(n)}(y)| \leq C|x-y|^{\kappa}.$$

$S$ is $\cC^\alpha$ 
if $S^\infty f^0 \in \cC^\alpha$ for any $f^0\in\liZ$.
$S$ is $\cC^{\alpha-}$ if $S$ is $\cC^{\alpha-\theta}$ for any $\theta>0$ small enough.
\end{rmk}

In fact, it turns out that  the regularity of $T$ in
Theorem \ref{teo:conv1}  does not limit the regularity of $S$. 
\begin{cor} \label{cor:conv2}
Suppose that $S$ is interpolatory, of the form \eqref{eq:nonlinearSS},
and  satisfies
\text{\textbf{C1, C2}} 
in Theorem \ref{teo:conv1}. Then $S$ is $\cC^\beta$ with $\beta= -\log_2(\eta)/L$.
\end{cor}
\begin{proof}
If $S$ is of the form \eqref{eq:nonlinearSS}, $S^{[1]}$ is well
defined. Notice that the expression for $S^{[1]}$ in
\eqref{eq:nonlinearS1}, together with the fact that $\cF$ satisfies
property \textbf{C1} in Theorem \ref{teo:conv1}, implies that there exists
$C_S>0 $ such that 
$\norma{S^{[1]}(f)}\leq C_S \norma{f}$.

Using $S^{[1]}$ we can   write
\begin{align*}
(Sf)_{2i+1}&= (Sf)_{2i+1}- (Sf)_{2i}+ (Sf)_{2i}
= \nabla (Sf)_{2i} + f_i = (S^{[1]} \nabla f)_{2i} + f_i \\
(Sf)_{2i+1}&= (Sf)_{2i+2} - (Sf)_{2i+2}+ (Sf)_{2i+1}
= f_{i+1} - \nabla (Sf)_{2i+1} = f_{i+1} - (S^{[1]}\nabla f)_{2i+1} \\
(Sf)_{2i+1}&= \frac12((S^{[1]} \nabla f)_{2i} + f_i) + \frac12(f_{i+1} - (S^{[1]}\nabla f)_{2i+1})
= \frac12 f_i + \frac12 f_{i+1} - \frac12 ((S^{[1]}\nabla f)_{2i+1} -(S^{[1]} \nabla f)_{2i}).
\end{align*} 
Hence, we can write
$$ Sf = T_{1,1} f + \cF_S(\nabla f), \qquad \cF_S(f)_{2i+1} =- \frac12 ((S^{[1]} f)_{2i+1} -(S^{[1]}  f)_{2i}),$$
where $T_{1,1}$ denotes the 2-point Deslauriers-Dubuc scheme.
 
By applying the same argument to $T_{n,n}$, the 2n-point
Deslauriers-Dubuc interpolatory 
subdivision scheme,  we can also write
$$ T_{n,n} = T_{1,1} + \cF_{n,n}(\nabla f).$$
Thus, we have too
$$ Sf = T_{n,n} f + (\cF_S-\cF_{n,n})(\nabla f).$$
Notice that $S^{[1]}$ and $T_{n,n}$ are bounded operators, hence
$\cF_S$ and $\cF_{n,n}$  and  $\cF_S - \cF_{n,n}$ satisfy
condition ${C1}$ in Theorem \ref{teo:conv1}. 

Since  $S$ satisfies condition \textbf{ C2}, its  regularity is  $\cC^\beta$,
$\beta= \min \{ - \frac{\log_{2}(\eta)}{L}, \alpha_n\}$, where
$\alpha_n$ is the regularity of $T_{n,n}$. But $n$ was taken
arbitrarily and
$\alpha_n\overset{n\rightarrow\infty}{\longrightarrow}+\infty$ (see
e.g.  \cite{Donoho92}), then $\beta= - \frac{\log_{2}(\eta)}{L}$.
\end{proof}

\section{ Linear subdivision schemes that reproduce Exponential
  Polynomials. Orthogonal Rules \label{sec:expo}}

The reproduction of certain finite dimensional spaces is always an
important issue in subdivision refinement. In particular, for linear interpolatory subdivision schemes there is a well known
relation between the smoothness of the scheme and the reproduction of 
spaces of polynomials \cite{Dyn92}. 

Conic sections, spirals and other objects  of
interest in geometric
modeling  can be   expressed as combinations of {\em exponential
polynomials}. The reproduction of such (finite-dimensional)
spaces, a very desirable property in CAGD, has been thoroughly studied
in the literature.  The study of
non-stationary subdivision schemes reproducing general spaces of
exponential polynomials was initiated in 2003 in \cite{DLL03}.
We recall that these spaces are associated to the space of solutions
of certain  linear differential operator and refer the reader to
\cite{CR11,DLL03,Ron88} for further information. Here, we
shall only recall the definitions and concepts relevant to our development.

 General  Exponential Polynomial (EP) spaces
are defined in terms of two sets of 
parameters, $\bar \gamma=(\gamma_0,\gamma_1,\ldots,\gamma_\nu)$,
$\gamma_i\neq \gamma_j$, $i \neq j$, $\gamma_i \in \mathbb{C}$ and 
$\bar \mu=(\mu_0,\mu_1,\ldots,\mu_\nu)$, $0\leq \mu_i \in \mathbb{N}$,
for some fixed $\nu\in\N$,
$\nu\geq 0$. In this paper, the associated space of exponential
polynomials  shall be denoted as follows
\begin{equation}\label{eq:exp-poli-gen}
V_{\bar \gamma}^{\bar \mu}= \espan{ t^l \exp(\gamma_n t) :
  l=0,1,\ldots,\mu_n-1, \, n=0,1,\ldots,\nu }.
\end{equation}
Notice that $V^{(1,1,1)}_{(0,\gamma,-\gamma)}=W_{0,\gamma}$, a notation 
 used mainly for simplicity throughout the paper.

We remark again that the set of parameters that define a
particular EP space needs to be known in order to determine the
linear, non-stationary, interpolatory subdivision scheme capable of
reproducing the space. In practice,  given an
initial set of samples, 
the parameters that determine the {\em correct} refinement rules are determined by a 
preprocessing step \cite{DLL03}.

As pointed out in \cite{DLL03} (Theorem 2.5), there
 exists a unique \emph{Minimal-Rank (M-R) reproducing} subdivision
 scheme for a given EP space of the form \eqref{eq:exp-poli-gen}.
Its mask can be derived from a linear system of equations (Theorem 2.3 of \cite{DLL03}) or
from the {\em orthogonal rules} of the EP space. These rules
annihilate samples of any function in the desired given EP space and give
rise to the so-called {\em orthogonal schemes}. 

We recall next  the definition of the
\emph{Z-transform} of a sequence $f=(f_i)_{i\in\Z}$ as the formal Laurent series
\begin{equation*} 
\ze{f}{z}=\sum_{i\in\Z} f_i z^i.
\end{equation*}
Note that $\ze{\cdot}{z}$ is a linear operator and that  $\ze{f}{z}$
is a well-defined function for $z\in\C$,
$0<|z|<1$, if $f\in\liZ$. In terms of the  Z-transform,  (\ref{eq:Sk-lin}) can be equivalently expressed  as follows,
\begin{equation*} 
\ze{T^{k}f}{z} = \ze{a^k}{z} \ze{f}{z^2}.
\end{equation*}
The Laurent polynomial $a^k(z):=\ze{a^k}{z}$ is the {\em symbol} of
$T^k$.

\begin{defi}{Orthogonal Rule.}
A Laurent series $b^k(z)=\sum_\iZ b_i^k z^i$ is an orthogonal rule for the
functional vector space  $\mathcal{V}$ at $\kZ$ if
\begin{equation*} 
b^k(z) \zek{k}{F}{z} = 0, \qquad \forall F\in\mathcal{V}, \quad
F^k:=F|_\kZ.
\end{equation*}
\end{defi}

One of the easiest examples of orthogonal rules 
is obtained from the fact that
  the $\mu$-th finite difference operator annihilates
equal-spaced samples of  $\Pi_{\mu-1}$, the space of  polynomials of
degree up to $\mu-1$. Using the Z-transform on the relation
$$ \nabla^{\mu} P^k = 0, \qquad \forall k\geq 0 , \quad \forall P\in
\Pi_{\mu-1}$$
we obtain
\begin{equation} \label{eq:pol_ort}
 z^{-\mu}(1-z)^{\mu} \zek{k}{P}{z} = 0.
\end{equation}
hence $b(z)=(1-z)^{\mu}$ is an orthogonal rule for $\Pi_{\mu-1}$ for any $\mu \geq 1$.
Using this result, we shall see next that we can easily obtain
orthogonal rules to various EP spaces.

\begin{teo} \label{thm:orto} Let $\bar \gamma\in \C^{\nu+1}$, $\bar \mu \in \N^{\nu+1}$,  with $\gamma_i\neq \gamma_j$ if $i\neq j$. Then
$$ \left( \prod_{n=0}^\nu (1-\exp(2^{-k}\gamma_n)z)^{\mu_n} \right)  \zek{k}{F}{z} = 0, \quad
\forall F\in V_{\bar \gamma}^{\bar \mu}.$$
\end{teo}
\begin{proof}
Consider first the  case $\nu=0$, $\gamma_0=\gamma$, $\mu_0=\mu$. In
this case, functions in $ V_\gamma^\mu$ are of the form  $F(t)=\exp(\gamma
t)P(t)$, with $P(t) \in \Pi_{\mu-1}$.
Then $F^k_i = \exp(2^{-k}i \gamma)P(2^{-k}i)$, and we have 
$$\zek{k}{F}{z} = \sum_\iZ \exp(2^{-k}\gamma i)P(2^{-k}i)z^i =
\sum_\iZ P(2^{-k}i) (\exp(2^{-k}\gamma)z)^i =
\zek{k}{P}{\exp(2^{-k}\gamma)z},$$
with $P^k=(P(2^{-k}i))_\iinZ$.
Replacing $z$ by $\exp(2^{-k}\gamma)z$ in (\ref{eq:pol_ort}) we obtain
$$ 0 = (1-\exp(2^{-k}\gamma)z)^{\mu} \zek{k}{P}{\exp(2^{-k}\gamma)z} = (1-\exp(2^{-k}\gamma)z)^{\mu} \zek{k}{F}{z}.$$

Now, for $\nu>0$, $F \in V_{\bar \gamma}^{\bar \mu}$ can be written as
$$F(t)=\sum_{n=0}^\nu F_n(t), \qquad F_n\in V_{\gamma_n}^{\mu_n}= \espan { t^l \exp(\gamma_n t) : l=0,\ldots,\mu_n-1}.$$

Since (by the case $\nu=0$) 
$$ (1-\exp(2^{-k}\gamma_n)z)^{\mu_n} \zek{k}{F_n}{z} = 0$$
we have
$$\left(\prod_{n=0}^\nu (1-\exp(2^{-k}\gamma_n)z)^{\mu_n}\right)
\zek{k}{F_m}{z}=0, \qquad m=0,\ldots, \nu$$
hence
$$0=\left( \prod_{n=0}^\nu (1-\exp(2^{-k}\gamma_n)z)^{\mu_n} \right)\sum_{m=0}^\nu \zek{k}{F_m}{z}=\left( \prod_{m=0}^\nu (1-\exp(2^{-k}\gamma_m)z)^{\mu_m} \right)\zek{k}{F}{z}.$$
\end{proof}

The next example shows how the orthogonal rules of an EP space
can be used to 
define a linear, non-stationary,  subdivision
scheme reproducing the given EP space. It also serves to examine 
a relation between the  non-stationary linear scheme obtained and
a {\em stationary} nonlinear rule, that is essentially
equivalent to the non-stationary refinement rules for data sampled in the given EP space. This derivation will be useful, in section \ref{sec:definition}, in the
construction of our new family of subdivision schemes.
 


\begin{ex} \label{ex:1}
Circles, hyperbolas and ellipses, centered at the origin,
can be drawn using functions in the EP space (for suitable values of
the parameter $\gamma$)
$$V_{(\gamma, -\gamma)}^{(1,1)}:=\espan{\exp{(\gamma t)},\exp{(-\gamma t)}}.$$

Using Theorem \ref{thm:orto}, we obtain that $\forall F \in V_{(\gamma,-\gamma)}^{(1,1)}$ 
\begin{align} \label{eq:ex:1:1}
0 &= (1-\exp(2^{-k}\gamma)z)(1-\exp(-2^{-k}\gamma)z) \zek{k}{F}{z}
= (1 - 2\phi_{\gamma,k} z + z^2) \zek{k}{F}{z}
 \end{align}
 with
 \begin{equation} \label{eq:phik-def}
 \phi_{\gamma,k} := \frac{1}{2} \left ( \exp(2^{-k} \gamma) +
   \exp(-2^{-k} \gamma) \right ) = \begin{cases}
\cosh(2^{-k} |\gamma|) , & \gamma\in\R \\
\cos(2^{-k} |\gamma|) , & \gamma\in\imath \R . 
\end{cases}
\end{equation}
Rewriting (\ref{eq:ex:1:1}) as 
 \begin{equation} \label{eq:ex:1:1b}
\sum_\iZ \left( F_{i}^{k} - 2\phi_{\gamma,k} F_{i-1}^{k} +
  F_{i-2}^{k}\right) z^i = 0, \end{equation}
and considering the coefficients of the even powers in the Laurent series, we get 
\begin{equation*} 
 F_{2i+2}^{k} - 2\phi_{\gamma,k} F_{2i+1}^{k} + F_{2i}^{k} = 0,
 \qquad \forall\iZ.
\end{equation*}
Then, since $F_{2i}^{k}=
F_{i}^{k-1}$, we can write
\begin{equation} \label{eq:coef0b} 
F_{2i+1}^{k} 
= \frac{1}{2\phi_{\gamma,k}}F_{i+1}^{k-1} +
\frac{1}{2\phi_{\gamma,k}}F_{i}^{k-1}.
\end{equation}
Thus, the linear,  
non-stationary, subdivision scheme
\begin{equation} \label{eq:ex:1:1.5}
 f^{k+1}_{2i}=f^{k}_i, \quad f^{k+1}_{2i+1} = \frac{1}{2\phi_{\gamma,k+1}}f^{k}_{i+1} + \frac{1}{2\phi_{\gamma,k+1}}f^{k}_{i}
 \end{equation}
reproduces $V_{(\gamma, -\gamma)}^{(1,1)}$, provided $\gamma$ is such
that $\phi_{\gamma,k} \neq 0$, $\forall k \geq 1$. This condition is obviously
fulfilled for any $\gamma \in \mathbb{R}$, and for $\gamma \in
\imath(-\pi,\pi)$. Note that for  $\gamma=0$ the scheme becomes 
$$
 f^{k+1}_{2i}=f^{k}_i, \quad f^{k+1}_{2i+1} = \frac{1}{2}f^{k}_{i+1} + \frac{1}{2}f^{k}_{i},
$$
i.e. the 2-point Deslauriers-Dubuc, $T_{1,1}$, subdivision scheme
(which reproduces $\Pi_1$).

Next, we derive a nonlinear stationary rule that is {\em essentially
  equivalent} to \eqref{eq:coef0b}, in a sense to be made precise
later. There are two key points in this derivation.
The first one is that \eqref{eq:ex:1:1b} implies that the parameters
$\phi_{\gamma,k}$ can be obtained from the level-$k$ samples of
functions $F \in V^{(1,1)}_{(\gamma,-\gamma)}$, provided that $F^k_i
\neq 0$, as follows
\begin{equation} \label{eq:ex-3}
 \phi_{\gamma,k} = \frac{F^k_{i-1} + F^k_{i+1}}{2 F^k_{i}}.
\end{equation}
 The second  key point is that the
definition of $\phi_{\gamma,k}$ in  \eqref{eq:phik-def} implies
that  the following two-scale relation holds true, 
\begin{equation} \label{eq:phi_quadrat}
\phi_{\gamma,k+1}^2=\frac12(1+\phi_{\gamma,k}), \qquad \forall k \geq 0.
\end{equation}
 Assuming that $\imath \gamma\in (-\pi,\pi)$ or $\gamma\in \R$,
 so  that $\phi_{\gamma,k+1}>0$, from  \eqref{eq:phi_quadrat} and
 \eqref{eq:ex-3} we get
$$ \phi_{\gamma,k+1}=\sqrt{\frac12(1+\phi_{\gamma,k})} =
\sqrt{\frac{F^k_{i-1} +2F^k_{i} + F^k_{i+1}}{4F^k_{i}}} .$$
Thus, from \eqref{eq:coef0b}  and the expression above,
we obtain the following nonlinear relation for the point-values of $F
\in V_{(\gamma, -\gamma)}^{(1,1)}$ (under appropriate restrictions so
that the operations involved may take place) 
$$  F^{k+1}_{2i+1} 
=\sqrt{\frac{F^k_{i}}{F^k_{i-1} +2F^k_{i} + F^k_{i+1}}}(F_i^k +
F_{i+1}^k), \qquad F\in V_{(\gamma, -\gamma)}^{(1,1)}. $$
This derivation serves as a motivation to consider the following
 set of  rules 
\begin{equation} 
\label{eq:Rgamma}
(Rf)_{2i}=f_i, \qquad (R f)_{2i+1}= \sqrt{\frac{f_{i}}{f_{i-1} +2f_{i} + f_{i+1}}}(f_i +
f_{i+1}), \qquad f\in\liZ
\end{equation}
as   {\em   stationary } representatives of the level-dependent rules
\eqref{eq:ex:1:1.5}. Obviously \eqref{eq:Rgamma} 
 is not always well-defined, due to the appearance of a square
root and a fraction. Hence, it does not define  a subdivision scheme,
but it does serve as a straightforward illustration of
the  issues that will be found in the next section, where we shall
carry out the  construction of a new family of subdivision
  schemes.

\end{ex}


\section{
A family of nonlinear, stationary, subdivision schemes with
EP-reproducing properties} \label{sec:definition}
 
Our goal is to define a symmetric 4-point stationary scheme that can
reproduce functions in $W_{0,\gamma}$, without any previous knowledge
of $\gamma$.
The starting point in our derivation  is the following (minimal-rank)
non-stationary, symmetric, 4-point subdivision scheme derived 
in \cite{DLL03}, here called $T_\gamma = (T_\gamma^k)_{k=0}^\infty$:
\begin{align}
(T_\gamma^k f^k)_{2i+1} 
&
=\frac12 f_i^k + \frac12 f_{i+1}^k - \Gamma_\gamma^k
\left( f_{i+2}^k-f_{i+1}^k - f_{i}^k + f_{i-1}^k  \right), \qquad
\Gamma^k_\gamma = \frac{1}{16} \phi_{\gamma,k+2}^{-2}
\phi_{\gamma,k+1}^{-1}, \label{eq:esquema_lineal}
\end{align}
where $\phi_{\gamma,k}$ is defined in \eqref{eq:phik-def}. 
This scheme is  studied in
\cite{DL95}, where it is shown that it is $\cC^{2-}$-convergent and reproduces the space of exponential polynomials
$V_{(0,\gamma,-\gamma)}^{(2,1,1)}$, 
$\gamma \in \R \cup \imath(-\pi,\pi)$.  Notice that for $\gamma=0$,
$\Gamma^k_\gamma= 1/16$ and $T_\gamma$ 
 becomes the (stationary) 4-point symmetric Deslauriers
Dubuc scheme.

We shall follow 
the path of the Example \ref{ex:1}
in order  to define
a  nonlinear 
(symmetric) 4-point rule, which is {\em stationary representative  } of the
subdivision rules 
\eqref{eq:esquema_lineal}. 
The first step is to use the two level relation \eqref{eq:phi_quadrat}
to  write $\Gamma^k_\gamma $ directly in terms of $\phi_{\gamma,k}$.
As in Example \ref{ex:1},  for 
 $\gamma\in\imath(-\pi,\pi)\cup \R$, 
 we  can write
\begin{align*}
8 \phi_{\gamma,k+2}^2 \phi_{\gamma,k+1}
= 4 (1+\phi_{\gamma,k+1}) \phi_{\gamma,k+1}
= (1+2\phi_{\gamma,k+1})^2-1
=\left(1+\sqrt{2(1+\phi_{\gamma,k})}\right)^2-1,
\end{align*}
hence
\begin{align*} 
\Gamma_\gamma^k &= \frac12 \frac{1}{\left(1+\sqrt{2(1+\phi_{\gamma,k})}\right)^2-1}.
\end{align*}
The second step is to consider the  orthogonal rules
of $W_{0,\gamma} = V_{(0,\gamma,-\gamma)}^{(1,1,1)}\subset
V_{(0,\gamma,-\gamma)}^{(2,1,1)}$ in order to relate the parameter
that defines the level dependent rules \eqref{eq:esquema_lineal} to
the functional samples.
Given $F\in  W_{0,\gamma} $, using Theorem \ref{thm:orto} we get
\begin{align*}
0&=(1-z)(1-2\phi_{\gamma,k}z+z^2) \zek{k}{F}{z} = \left(1-(2\phi_{\gamma,k}+1)z+(2\phi_{\gamma,k}+1)z^2-z^3\right) \zek{k}{F}{z},
\end{align*}
or, equivalently, 
\begin{equation}
F^k_{i-1} -(2\phi_{\gamma,k}+1)F^k_{i} +(2\phi_{\gamma,k}+1) F^k_{i+1} -F^k_{i+2}= 0, \qquad \forall \iZ,\label{eq_nul}
\end{equation}
from which we get that
\begin{equation}
 \phi_{\gamma,k} =\frac12 (\frac{F^k_{i+2} - F^k_{i-1}}{F^k_{i+1}-F^k_i}-1),
 \qquad \forall \iZ \quad s.t.\quad F^k_i \neq F^k_{i+1}. \label{eq:cos}
 \end{equation}
Notice that if $F^k_i =F^k_{i+1}$,  from (\ref{eq_nul}) we deduce that
$F^k_{i+2}=F^k_{i-1}$. For these values of $i$,  we must be sampling
around an extremum, or a flat region,  of $F(t)$ and, in this case,
$\phi_{\gamma,k}$ 
cannot be deduced from the data $F^k_l$, $l=i-1,\ldots, i+2$. On the
other hand, \eqref{eq:cos} implies that $\phi_{\gamma,k}$ can always
be determined from the 
available data when the 
sampled points belong to a region in which $F$ is strictly monotone.

We remark that if the orthogonal rules of
$V_{(0,\gamma,-\gamma)}^{(2,1,1)}$ 
are used to express
$\phi_{\gamma,k}$ in terms of the functional samples $F^k$, the
associated nonlinear rule will involve  5 functional 
samples. We have used the orthogonal rules of $W_{0,\gamma}$ to obtain
a nonlinear rule involving only 4 points, as in \eqref{eq:esquema_lineal}.

 Let us, thus, consider that $F^k_i \neq F^k_{i+1}$. Then, we readily
 deduce from \eqref{eq:cos} that
\begin{equation*} 
2( 1+\phi_{\gamma,k} )= \frac{F^k_{i+2} -
    F^k_{i-1}}{F^k_{i+1}-F^k_i} +1 \quad \Rightarrow \quad \Gamma_\gamma^k
=\frac12 \frac{1}{\left(1+
\sqrt{1+\frac{F^k_{i+2} - F^k_{i-1}}{F^k_{i+1}-F^k_i}} \right)^2-1}.
\end{equation*}

The observations above lead us to consider the function
\begin{equation} \label{eq:Gamma-def}
 \Gamma(f_{-1},f_0,f_1,f_2)
:=\frac12 \frac{1}{\left(1+\sqrt{1+ \frac{f_2 - f_{-1}}{f_1-f_0}}\right)^2-1}, 
\end{equation}
which satisfies 
$ \Gamma^k_\gamma = \Gamma(F^k_{i-1}, F^k_i, F^k_{i+1}, F^k_{i+2})$, $\forall i\in\Z$ s.t. $F^k_i \neq
F^k_{i+1}$, when $F^k= F|_{\kZ}$, $F \in W_{0,\gamma}$. 
Hence, the nonlinear stationary rule
\begin{equation} \label{eq:primerS}
(R f^k)_{2i+1} =\frac12 f_i^k + \frac12 f_{i+1}^k - \Gamma(f_{i-1}^k,f_i^k,f_{i+1}^k,f_{i+2}^k) \left(f_{i+2}^k-f_{i+1}^k - f_{i}^k + f_{i-1}^k  \right),
\end{equation}
satisfies $(R F^k)_{2i+1}=F^k_{2i+1}$ if $F^k= F|_{\kZ}$, $F \in
W_{0,\gamma}$ in strictly monotone regions of $F$, since in this case the application of \eqref{eq:primerS} is equivalent to the application of \eqref{eq:esquema_lineal}.

Since   $\Gamma$ in \eqref{eq:Gamma-def} is only well-defined if $f_1
\neq f_0$ and $1+(f_2-f_{-1})/(f_1-f_0) >0$, \eqref{eq:primerS} cannot
be directly applied to general 
sequences $f \in l_\infty(\Z)$.  In order to have a
nonlinear rule applicable to all sequences in $\liZ$, we
propose to 
introduce the cut-off
function $\Gamma_{\epsilon}$ defined  as follows ($\epsilon>0$):
\begin{equation} \label{eq:GammaE-def}
\Gamma_\epsilon (f_{-1},f_0,f_1,f_2) = \begin{cases}
\frac12 \frac{1}{\left(1+\sqrt{ 1+\frac{f_2 -
        f_{-1}}{f_1-f_0}}\right)^2-1}, & f_1\neq f_0 \text{ and }
1+\frac{f_2 - f_{-1}}{f_1-f_0} \geq \epsilon^2\\
0, & 
f_{-1}\leq f_0 =f_1\leq f_2 \text{ or } f_{-1}\geq f_0=f_1 \geq f_2 \\
\frac{1}{16}, & \text{ otherwise}.
\end{cases} 
\end{equation}

\begin{defi} 
Given $\epsilon> 0$, an interpolatory, nonlinear, stationary subdivision scheme
$S_\epsilon$ is defined by the following insertion rule 
\begin{equation} \label{eq:seps-def}
(S_\epsilon f)_{2i+1}=\frac12 f_i + \frac12 f_{i+1} - \Gamma_\epsilon(f_{i-1},f_i,f_{i+1},f_{i+2})\left(f_{i+2}-f_{i+1} - f_{i} + f_{i-1}  \right).
\end{equation}
\end{defi}

In the definition of the function $\Gamma_\epsilon$ in
\eqref{eq:GammaE-def},  we have taken into account the following considerations:

1. $\Gamma_\epsilon: \R^4 \rightarrow \R$ must be a bounded
function. This will be an important ingredient in proving the
convergence of the family of schemes \eqref{eq:seps-def}.
Indeed
\[ 1+\frac{f_2 - f_{-1}}{f_1-f_0} \geq \epsilon^2 \quad \rightarrow
\quad \Gamma_\epsilon(f_{-1},f_0,f_1,f_2) \leq 
\frac12 \frac{1}{(1+\epsilon)^2 -1 }  =\frac{1}{2 \epsilon (2+\epsilon)}=:M_\epsilon. \]
  In what follows, we shall  assume that $0<\epsilon\leq 2$, so that
  $\Gamma_\epsilon(f_{-1},f_0,f_1,f_2) \leq \max \left \{
    \frac{1}{16},  M_\epsilon \right \}=M_\epsilon $.

2. If $f_0=f_1$ but we are on a 'monotone' region, we seek  to generate
monotone data. For this reason, the definition of $\Gamma_\epsilon$ in
this case 
leads to $S_\epsilon \equiv T_{1,1}$.

3. If none of the above holds,  we revert to the
linear $T_{2,2}$ prescription (maximal regularity). 

Next, we state and prove  the reproduction properties of the schemes
$S_\epsilon$. As mentioned in the
introduction, by {\em functional reproduction} we mean  the capability of a subdivision scheme to construct (or recover) a
 particular function from its point evaluations at the integers. 
As observed in \cite{CR11}, Definition 
\ref{defi:rep_gen} is equivalent to the following step-wise reproduction  property
\begin{equation}  \label{eq:rep}
F^{k+1} = S^k F^k, \qquad \forall F\in\mathcal{V},
\end{equation}
with $F^k:= F|_\kZ$, as long as $S$ is {\em
  non-singular}\footnote{ $ S$ is non-singular if  $S^\infty f^0=0 \,
  \leftrightarrow f^0=0$.} 
  and either linear or interpolatory. The equivalence for linear
  schemes was proved in \cite{CR11}, but it also holds
  for  interpolatory   schemes (linear or not).

\begin{lem}
Let $S$ be an interpolatory subdivision scheme. $S$ satisfies the
step-wise reproduction property \eqref{eq:rep} if and only if  $S$ reproduces
the functional space $\mathcal{V}$, in the sense of Definition \ref{defi:rep_gen}.
\end{lem} 
\begin{proof}
   Assume $S$
   satisfies   \eqref{eq:rep}. If $f^0=F|_\Z$, obviously
   $f^k=F^k=F|_\kZ$ and (by the definition of convergence)  $S^\infty
   f^0 = F$.  For the opposite 
  direction, since the scheme is interpolatory  $S^k f^k= (S^\infty
  f^0)(\kZ)=F(\kZ)$,  which implies step-wise reproduction.
\end{proof}

Thus, we use \eqref{eq:rep} to check the reproduction properties of
$S_\epsilon$. Reproduction of $\Pi_2$ is almost immediate.

\begin{prop} \label{prop:repro_poli2}
 $S_\epsilon$ reproduces $\Pi_2$ for $\epsilon\in [0,2]$.
\end{prop}
\begin{proof}
Let $F \in \Pi_2$, and $F^k=F|_\kZ$. Since $\nabla^3 F^k \equiv 0$, we have that
$$F^k_{i-1} -3F^k_{i} +3 F^k_{i+1} -F^k_{i+2}= 0, \qquad \forall  i \in \Z. $$
For a given $ i \in \Z$, we consider all the possible cases:   If
  $F^k_i \neq F^k_{i+1}$, from the previous relation
$$ 
\frac{F^k_{i+2} - F^k_{i-1}}{ F^k_{i+1} - F^k_{i}}=3  \quad\Longrightarrow
 \quad  \Gamma_\epsilon(f_{i-1},f_i,f_{i+1},f_{i+2})= \frac{1}{16},$$
 which implies that $(S_\epsilon F^k)_{2i+1}  = (T_{2,2} F^k)_{2i+1}
 = F^{k+1}_{2i+1}$, because $T_{2,2}$ reproduces $\Pi_3$.
If  $F^k_i = F^k_{i+1}$,  and
\begin{equation} \label{eq:Fk_monotone}
F^k_{i-1}\leq F^k_i = F^k_{i+1} \leq F^k_{i+2} \text{ or }
F^k_{i-1}\geq F^k_i = F^k_{i+1} \geq F^k_{i+2},
\end{equation}
then $(S_\epsilon F^k)_{2i+1} = (T_{1,1} F^k)_{2i+1}$. But if $F\in
\Pi_2$ and satisfies \eqref{eq:Fk_monotone}, it must be a constant function,
hence $ (T_{1,1} F^k)_{2i+1}= F^{k+1}_{2i+1}$.
If  $F^k_i = F^k_{i+1}$ but \eqref{eq:Fk_monotone} does not hold,
then
$(S_\epsilon F^k)_{2i+1}=(T_{2,2} F^k)_{2i+1}= F^{k+1}_{2i+1}$. 
\end{proof}
For functions in $W_{0,\gamma}$ we have the following result.

\begin{prop} \label{prop:repro0}
Let $\epsilon\in [0,2]$, and 
$F\in W_{0,\gamma}$. Then 
\begin{align*}
 \text{If } \quad F_i^k \neq F_{i+1}^k  \quad
&\Longrightarrow \quad  (S_\epsilon F^k)_{2i+1} = F^{k+1}_{2i+1},
\qquad \forall k\geq0, \forall  \iZ,
\end{align*}
provided that $\gamma\in\R$ or $\gamma\in \imath (-\pi,\pi)$ s.t. $\cos(|\gamma|) \geq -1+\epsilon^2/2$.
\end{prop}
\begin{proof}
Let  $F\in W_{0,\gamma}$ and $i, k$ such that
$ F_i^k \neq  F_{i+1}^k$. Using \eqref{eq:cos} 
$$\phi_{\gamma,k} = \frac12 \left (\frac{F^k_{i+2} -
  F^k_{i-1}}{F^k_{i+1}-F^k_i}-1 \right ) \quad \Rightarrow \quad 
\frac12 ( 1+\phi_{\gamma,k} )= \frac14 \left ( \frac{F^k_{i+2} -
    F^k_{i-1}}{F^k_{i+1}-F^k_i} +1 \right ).$$
If we can prove that $\Gamma_\epsilon(F^k_{i-2},F^k_{i-1},F^k_{i},F^k_{i+1}) =
\Gamma^k_\gamma$, we readily obtain 
$(S_\epsilon F^k)_{2i+1}=(T_\gamma^k F^k)_{2i+1}=F^{k+1}_{2i+1}$, from
the EP reproduction properties of $T_\gamma^k$. 

If $\gamma\in\R$ then $\phi_{\gamma,k}=\cosh(2^{-k}
|\gamma|) \geq 1$,  hence $\frac12( 1+\phi_{\gamma,k} ) \geq
\epsilon^2/4$ $\forall \epsilon\in [0, 2]$.
If $\gamma\in
\imath(-\pi,\pi)$ and  $\cos(|\gamma|) \geq -1+\epsilon^2/2$, we also have that
$\frac12( 1+\phi_{\gamma,k} ) \geq \epsilon^2/4$.
 Hence, in both cases, 
$$  \frac12( 1+\phi_{\gamma,k} ) \geq \epsilon^2/4 \longrightarrow 1+\frac{F^k_{i+2} -   F^k_{i-1}}{F^k_{i+1}-F^k_i} \geq
\epsilon^2  \longrightarrow  \Gamma_\epsilon(F^k_{i-2},F^k_{i-1},F^k_{i},F^k_{i+1}) =
\Gamma^k_\gamma.$$
\end{proof}
From a practical point of view, the parameter $\epsilon$ restricts the
values of $\gamma$ for which 
the space $W_{0,\gamma}$ is reproduced by $S_\epsilon$.
The  previous result shows that  the value of $\epsilon \in [0,2]$
does not affect the 
  reproduction of hyperbolic  functions, i.e. $W_{0,\gamma}$, for
  $\gamma \in \R$,  but restricts
  the reproduction of trigonometric functions. However, the
  restriction in Proposition
  \ref{prop:repro0} is not   too severe. For $\epsilon=1$  and values
  of  $\gamma \in \imath (-\pi,\pi)$  such that
  $\cos(|\gamma|)\geq -1/2$, Proposition \ref{prop:repro0} holds.  In
  particular, since $\cos(2\pi/3)=-1/2$, circles can
  be reproduced 
  from $n$ samples for $n\geq 3$,  as long
  as the initial samples  satisfy $F^0_i\neq F^0_{i+1}$, $\forall 
  \iZ$. 

As an example, let us consider the function
\begin{equation*} 
F_u(t)=\cos\left (\frac{2\pi}{3} t +u\right ), \qquad u \in \mathbb{R}.
\end{equation*}
In Figure  \ref{fig:rep_cercle_min}
we apply our scheme (with $\epsilon=1$) to refine the initial set $f^0
= (G_u(i))_\iinZ$  where  $G_u(t)  = (\cos(\gamma t+u), \sin(\gamma
t+u))$, with $\gamma=2 \pi/3$ 
 i.e. three points in a circle. Since $ \cos(2\pi/3) =\cos(4\pi/3)$, the
  x-components of $G_0(1)$ and $G_0(2)$ are equal.   The condition
  $F_{0,i}^0 \neq F_{0,i+1}^0 
\forall i $ is not
satisfied and   the circle is not reproduced. On the other hand, by
slightly modifying the value of $u$, for 
instance to $u=10^{-5}$, the condition that ensures exact reproduction
is fulfilled. As a consequence, the circle is correctly reproduced. In
Figure \ref{fig:rep_cercle_min}, the limit curve obtained with the
4-point Deslauriers-Dubuc scheme is also shown for comparison. 

\begin{figure}[!h]
\centering
\includegraphics[clip, width=0.5 \textwidth]{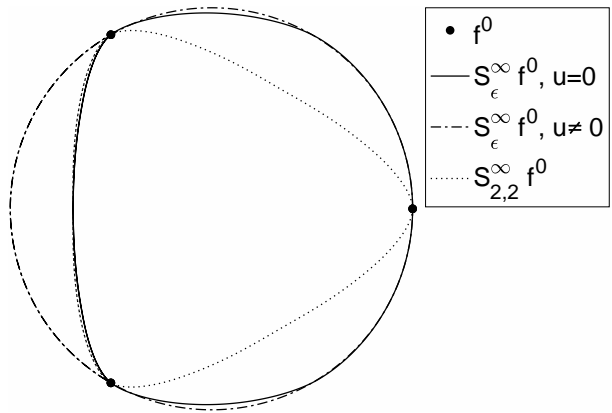}
\caption{ \label{fig:rep_cercle_min} Curves generated by  $S_\epsilon$
  ($\epsilon=1$)  and $T_{2,2}$  from an initial sequence
  $f^0 = (G_u(i))_\iinZ$, where  $G_u(t)
  = (\cos(\gamma t+u), \sin(\gamma t+u))$, for $u=0$ and
  $u=u=10^{-5}$. Observe that  $S^\infty_\epsilon f^0$ strongly depends on
  the value of $u$.}
\end{figure}
 
This example shows also that $S_\epsilon$ is convergent
(at least for $\epsilon=1$), a
fact that will be proven in the next section, but it is not stable.  Nevertheless, in section \ref{sec:stable} we show that
stability holds  if the initial data is appropriately restricted.

\section{Convergence} 
 \label{sec:convergence}}

In this section we check that the schemes in \eqref{eq:seps-def} are
of the form \eqref{eq:nonlinearSS}. Then we prove  convergence by
checking the conditions in 
Theorem \ref{teo:conv1}. 
Observe that
$$\frac{f_2 - f_{-1}}{f_1-f_0}
= \frac{\nabla f_1 + \nabla f_0 + \nabla f_{-1}}{\nabla f_0} \quad \longrightarrow \quad \Gamma_\epsilon(f_{-1},f_0,f_1,f_2) = \Gamma_\epsilon^{[1]}(\nabla f_{-1},\nabla f_0,\nabla f_1)$$
with
\begin{equation*} 
\Gamma^{[1]}_\epsilon (f_{-1},f_0,f_1) = \begin{cases}
\frac12
\frac{1}{\left(1+\sqrt{1+\frac{f_{-1}+f_0+f_1}{f_0}}\right)^2-1}, &
f_0\neq 0 \text{ and } 1+\frac{f_{-1}+f_0+f_1}{f_0} \geq  \epsilon^2 \\
0, & f_{-1} \cdot f_1 \geq 0 = f_0\\
\frac{1}{16}, & \text{ otherwise.}
\end{cases}
\end{equation*}
Hence, we can write
\begin{align} \label{eq:Sepsform10}
(S_\epsilon f)_{2i+1} &= \frac12 f_i + \frac12 f_{i+1} - \Gamma_\epsilon^{[1]}(\nabla f_{i-1},\nabla f_i,\nabla f_{i+1}) \left(\nabla f_{i+1} - \nabla f_{i-1}  \right), 
\end{align}
that is, $S_\epsilon$ is of the form \eqref{eq:nonlinearSS} with
$T= T_{1,1}$ and $\cF_\epsilon : \liZ \rightarrow \liZ$ given by
\begin{equation} \label{eq:Sepsform10F}
 \cF_\epsilon(f)_{2i}= 0, \qquad \cF_\epsilon(f)_{2i+1}=  \Gamma_\epsilon^{[1]}(
f_{i-1}, f_i, f_{i+1}) \left( f_{i-1} -  f_{i+1}  \right). 
\end{equation}
\begin{teo} \label{teo:converg}
 The  scheme  $S_\epsilon$ is convergent for $\epsilon\in(\sqrt{3}-1,2]$.
\end{teo}
 
\begin{proof} 
By \eqref{eq:nonlinearS1}, $S_\epsilon^{[1]}=T_{1,1}^{[1]} +
\nabla \cF_\epsilon$, i.e.
$$
(S_\epsilon^{[1]}f)_{2i+j}= \frac12 f_i +(-1)^j \Gamma^{[1]}_\epsilon(f_{i-1},f_i,f_{i+1}) (f_{i-1} - f_{i+1} )  ,  \qquad j=0,1.
$$

We check  conditions {\bf C1.} and
{\bf C2.} in  Theorem \ref{teo:conv1}. 
For this, we distinguish the following cases:

\begin{enumerate}

\item $f_i\neq 0$ and $1+ (f_{i-1}+f_i+f_{i+1})/f_i \geq \epsilon^2$.
In this case we have that
$$
0 \leq \Gamma_\epsilon^{[1]} (f_{i-1},f_i,f_{i+1}) \leq  M_\epsilon 
=\frac{1}{2\epsilon(2+\epsilon)},
$$
hence,
\begin{align*}
| \cF(f)_{2i+1}| \leq 2 M_\epsilon ||f||_\infty, \qquad  |(S_\epsilon^{[1]}f)_{2i+j}| \leq ||f||_\infty ( \frac12 + 2 M_\epsilon).
\end{align*}

\item  $f_i=0$, $f_{i+1} \cdot f_{i-1} \geq 0$. Then
\[ \cF(f)_{2i+1}=0, \qquad (S^{[1]}f)_{2i+j}= \frac12 f_i \quad \rightarrow \quad
|(S^{[1]}f)_{2i+j}| \leq \frac12 ||f||_\infty, \]

\item otherwise
\[ | \cF(f)_{2i+1}| \leq \frac18 ||f||_\infty, \qquad (S^{[1]}f)_{2i+j}= \frac12 f_i -\frac{1}{16} (f_{i-1}-f_{i+1})
\quad \rightarrow \quad |(S^{[1]}f)_{2i+j}| \leq \frac58
||f||_\infty. \]
\end{enumerate}
Then, condition {\bf C1.} clearly  holds for any value of
$\epsilon$. In addition, since $0<\epsilon\leq 2$,  condition {\bf
  C2.}  is satisfied with $L=1$ for those values of $\epsilon$  satisfying  
\[  \max \{\frac58, \frac12 +   2 M_\epsilon\}=\frac12 + 2 M_\epsilon
< 1 \quad \rightarrow \quad 
\epsilon>\sqrt{3}-1 \approx 0.7321. \]  
\end{proof}
Taking into account  Corollary \ref{cor:conv2}, $S_\epsilon$ is at least $\cC^{\beta-}$,
 $\beta=-\log_2(\frac12 + \frac{1}{\epsilon(2+\epsilon)})$. This is a
  conservative result, since $\beta \in (0,0.5)$ for $\epsilon \in
  (\sqrt{3}-1, 2]$. The question of smoothness will be considered in more
  detail in sections \ref{sec:smoothness} and \ref{sec:numeric}.

On the other hand, we observe that for $\epsilon=\sqrt{3}-1$, the
restriction in Proposition \ref{prop:repro0} is $\cos(|\gamma|)\geq 1-\sqrt{3}$,
which is satisfied when $\gamma \in
\imath[-\frac34\pi,\frac34\pi]$, as stated in the introduction.

\section{Monotonicity preservation} \label{sec:mono}

In applications dealing with increasing (or decreasing) sequences of
data, it  is often convenient to maintain this feature after recursive
refinement. In this section we examine the conditions that guarantee
that the family of nonlinear schemes in \eqref{eq:nonlinearSS} have
this property.

\begin{defi}
$S$ is \emph{monotonicity preserving} if
$$ \nabla f^k_i \geq 0, \, (\leq 0) \quad \forall i\in\Z \quad \lra \quad
\nabla f^{k+1}_i\geq 0, \, (\leq 0) \quad \forall i\in\Z.$$
It is \emph{strictly monotonicity preserving} if the above
relations  hold with strict inequalities.
\end{defi}
\begin{teo} \label{teo:mono}
$S_\epsilon$ is (strictly) monotonicity preserving for $\epsilon \in [0 \sqrt{2}]$.
\end{teo}
\begin{proof}
 Since $\nabla f^{k+1} = \nabla S_\epsilon f^k = S_\epsilon^{[1]}
 \nabla f^k$, and $S^{[1]}(\alpha
f)=\alpha S^{[1]}(f)$ $\forall \alpha \in \R$,  proving that $S_\epsilon$ is monotonicity preserving is
 equivalent to proving that $S^{[1]}_\epsilon$ is {\em positivity
   preserving}, i.e.
$$ f_i\geq0,\, \forall i\in\Z \quad \lra \quad S_\epsilon^{[1]}f_i\geq 0, \, \forall i\in\Z, \qquad \forall f\in\liZ.$$

First, we notice that if $f_i \geq 0$ $\forall i\in\Z$ and 
 $\epsilon \leq \sqrt{2}$
\begin{align} \label{eq:Gamma1+}
\Gamma^{[1]}_\epsilon (f_{i-1},f_i,f_{i+1}) &=
\begin{cases}
\frac12 
\frac{1}{\left(1+\sqrt{1+\frac{f_{i-1}+f_i+f_{i+1}}{f_i}}\right)^2-1}\,
,& f_i\neq 0 \\
0 \, , & f_i= 0.
\end{cases}
\end{align}

Thus, if $f_i =0$, $(S_\epsilon^{[1]}f)_{2i+j}=0$, $j=0,1$, while for 
 $f_i>0$ we can write
\[ \Gamma^{[1]}_\epsilon(f_{i-1},f_i,f_{i+1}) (f_{i-1} - f_{i+1} )=
\frac{f_i}{2}\frac{f_{i-1}/f_i
    - f_{i+1}/f_i}{(\sqrt{2+f_{i-1}/f_i + f_{i+1}/f_i}+1)^2-1}=
\frac{f_i}{2} H(f_{i-1}/f_i,f_{i+1}/f_i )\]
so that
\begin{equation} \label{eq:S_hache}
(S_\epsilon^{[1]}f)_{2i+j}=
\frac12 f_i \left (1+(-1)^j H(f_{i-1}/f_i,f_{i+1}/f_i) \right ),
\qquad j=0,1 ,
\end{equation}
with
\begin{equation} \label{eq:H}
H(x,y):=\frac{x - y }{(\sqrt{2+x + y}+1)^2-1} = \frac{x-y}{2+x + y +2\sqrt{2+x + y}}.
\end{equation}
By straightforward algebra, we have
$$-1 < \frac{-x-y}{2+x + y +2\sqrt{2+x + y}} \leq H(x,y)\leq
\frac{x+y}{2+x + y +2\sqrt{2+x + y}}<1,
$$
so that $H(x,y)\in (-1,1)$, $\forall x,y \geq 0$. This concludes the proof.

\end{proof}

\begin{rmk} \label{rmk:pos_monotone}
For strictly positive (negative) data and  $\epsilon \in [0,
\sqrt{2}]$, $\Gamma_\epsilon^{[1]}(f_{i-1},f_i,f_{i+1})$ is given by
\eqref{eq:Gamma1+} (notice that it is
independent of the value of $\epsilon$). It can be easily seen that,
in this case,
\[ |\Gamma_\epsilon^{[1]}(f_{i-1},f_i,f_{i+1})| \leq
\frac{1}{2+2\sqrt{2}} \]
hence   
\begin{align*}
|(S_\epsilon^{[1]}f)_{2i+j}|&
 \leq (\frac12 +| 2\Gamma_\epsilon^{[1]}(f_{i-1},f_i,f_{i+1})
|)\|f\|_\infty \leq (\frac12 +
\frac{1}{1+\sqrt{2}})\|f\|_\infty. 
\end{align*}
Condition \text{\textbf{C2.}} in Theorem
\ref{teo:conv1} is, thus, fulfilled with $L=1$ and
$\eta=\frac12+\frac{1}{1+\sqrt{2}}<1$. 
 Hence, if $f^0$ is a strictly monotone sequence, and  $\epsilon \in
 [0,\sqrt{2}]$, 
$S_\epsilon^\infty f^0$ is a continuous function that preserves the  monotonicity
properties of the initial data. This function does not depend on the
value chosen for the parameter $\epsilon$. 
 \end{rmk}

\section{Smoothness} \label{sec:smoothness}

In this section we examine the smoothness of the schemes
$S_\epsilon$. We shall see first that $S_\epsilon$ admits a \emph{first divided difference scheme}, which relates the
divided differences of the data produced by $S_\epsilon$ at
consecutive resolution levels.

 Divided    differences at level $k$ are
defined as  $d^k_i :=
2^k \nabla    f^k_i, \,   \iinZ$, where $f^k= S_\epsilon f^{k-1}$, $k\geq 1$.
Since $S^{[1]}_\epsilon $  satisfies $2 S^{[1]}_\epsilon f= S^{[1]}_\epsilon 2f$,
$\forall f \in \liZ$,   we can write
$$ 2^{k+1}\nabla f^{k+1} = 2 S^{[1]}_\epsilon 2^{k}\nabla f^{k} \quad
\Longrightarrow  d^{k+1} = 2S^{[1]}_\epsilon d^{k}$$
so that   the \emph{first divided difference scheme} can be  defined as
$S_\epsilon^{(1)}:=2S_\epsilon^{[1]}$. 
It is well known (see  e.g. \cite{Dyn92}) that if $S_\epsilon^{(1)}$
is a convergent subdivision scheme, then $S_\epsilon$ 
converges to $\cC^1$ functions and
\begin{equation*} 
((S_\epsilon^{(1)})^\infty d^0)(t) = \frac{d}{dt} (\Sinf_\epsilon f^0)(t).
\end{equation*}

In this section we study the conditions that guarantee convergence of the subdivision scheme $S_\epsilon^{(1)}$.
 As in \cite{Dyn92,Kuijt98}, for a given $d^0$ we shall define 
$\mathbb{P}^{k}(t)$ as  the continuous 
piecewise linear function such that 
\begin{equation} \label{eq:Pkdef}
 \mathbb{P}^{k}(i2^{-k}) = d^{k}_i, \qquad \forall \iinZ
\end{equation}
$d^k=S^{(1)}_\epsilon d^{k-1}$, $k\geq 1$,  and study the conditions
that ensure that $(\mathbb{P}^{k})_{k\geq 0}$ is a Cauchy sequence. 

Observe that 
$$ \norma{ \mathbb{P}^{k+1} -  \mathbb{P}^{k}} = \sup_\iZ
\max\{|d^{k+1}_{2i} - d^{k}_i| , |d^{k+1}_{2i+1} -
\frac12(d^{k}_i+d^{k}_{i+1})|\} 
$$
and that, using that $d^k_i = \frac12(d^{k+1}_{2i}+d^{k+1}_{2i+1})$,
\begin{align*}
|d^{k+1}_{2i} - d^{k}_i| &= \frac12|d^{k+1}_{2i+1} - d^{k}_{2i}| \leq \frac12 \norma{\nabla d^{k+1}}\\
|d^{k+1}_{2i+1} - \frac12(d^{k}_i+d^{k}_{i+1})| &= |d^{k+1}_{2i+1} - \frac14(d^{k+1}_{2i}+d^{k+1}_{2i+1}+d^{k+1}_{2i+2}+d^{k+1}_{2i+3})|
\\
&= \frac14 |d^{k+1}_{2i+1} - d^{k+1}_{2i}+ 2(d^{k+1}_{2i+1} - d^{k+1}_{2i+2})+d^{k+1}_{2i+2} - d^{k+1}_{2i+3}| \leq \norma{\nabla d^{k+1}}.
\end{align*}
As a consequence, we have that $\norma{ \mathbb{P}^{k+1} -
  \mathbb{P}^{k}} \leq \norma{\nabla d^{k+1}}$, so that the
convergence of $S_\epsilon^{(1)}$ follows from proving that $\norma{\nabla d^{k}} \overset{k\to\infty}{\lra} 0$.
Notice that for $d^k_i\neq 0$
$$ d^k_{i+1} - d^k_i = d^k_i \left ( \frac{ d^k_{i+1}}{d^k_i} - 1
\right )$$
hence, we might prove the desired result by checking the following two
conditions 
\begin{equation} \label{eq:cond0}
\text{(a) } \quad \norma{d^k}<M<\infty, \quad \forall k\geq 0 \qquad
\text{(b) } \quad \lim_{k\to\infty} \sup_\iinZ | \frac{
  d^k_{i+1}}{d^k_i} - 1|= 0,
\end{equation}
at least for a restricted class of initial data $d^0$ such that
$d^0_i\neq 0$, $\forall \iinZ$.

Notice that if the initial data satisfy $d^0_i>0(<0)$, $\forall\iinZ$
and $\epsilon \in [0,\sqrt{2}]$, then   Theorem
\ref{teo:mono} ensures that $d^k_i\neq 0$ $\forall i, k$. Since 
 the repeated application  
of $S_\epsilon$ preserves strict monotonicity, we have that
 $d^k_i>0 (<0)$, $\forall k\geq0 \,
\forall\iinZ$.

Hence, throughout this section we shall assume that we have strictly monotone
data and we shall silently assume that $\epsilon \in [0,\sqrt{2}]$. In  Remark
\ref{rmk:pos_monotone} we observed that, in this case,
$S_\epsilon$ does not depend on the value of $\epsilon$
and is convergent. 
From 
\eqref{eq:S_hache} we get  the expression of $S_\epsilon^{(1)}$
 in this case,  
\begin{equation} \label{eq:S1expression}
 (S_\epsilon^{(1)} d)_{2i+j}= \Psi_j(d_{i-1},d_{i},d_{i+1}) :=
 d_i\left (1+(-1)^j H(\frac{d_i}{d_{i-1}},\frac{d_{i+1}}{d_i})\right ) , \qquad j=0,1,
\end{equation}
where $H(x,y)$ is defined in \eqref{eq:H}. Notice that  $\Psi_j$, $j=0,1$, are
1-homogeneous functions. In addition, the results
obtained in the previous section ensure that 
$\Psi_j: \R^3_+ \rightarrow \R_+$, $j=0,1$, ($R_+=(0,+\infty)$), i.e. they
are {\em positive functions}, when restricted to positive data. They
are also smooth, being  compositions of
smooth functions that are always well-defined for positive data. 

We shall examine next what are the conditions to ensure that  
\eqref{eq:cond0}-(b) is satisfied. As in \cite{Kuijt98}, given an
initial sequence $d^0$, with $d^0_i >0 (<0)$, and $d^k=
S_\epsilon^{(1)} d^{k-1}$, $k\geq 1$, we define
 \begin{equation} \label{eq:rho}
 R^k_i:= \frac{d^k_{i+1}}{d^k_i}, \quad r^k_i = 1/R^k_i, \qquad 
\rho^k:=\sup_{i\in\Z}\left\{\left|r^k_{i}-1\right|,\left|R^k_i-1\right|\right\},
\end{equation}
and study the
behavior of the sequence $\rho^k$.  Obviously, proving that $\lim_{k
  \rightarrow \infty} \rho^k =0$ leads to \eqref{eq:cond0}-(b). As a previous
  step,  we need to restrict the
class of initial data to ensure that $\rho^k <+\infty, \quad \forall
k\geq 0$. 

\begin{lem} Let $d^k \in \liZ$, $d^k_i>0 (<0)$, and $d^{k+1}=
S_\epsilon^{(1)} d^{k}$, $k\geq 0$. Then, with the definitions in
\eqref{eq:S1expression}-\eqref{eq:rho}, if $\rho^k < +\infty$, then $\rho^{k+1}<+\infty$ 
\end{lem}
\begin{proof}
Using  \eqref{eq:S1expression} and \eqref{eq:rho} we can write
$$R^{k+1}_{2i} =\frac{\Psi_1(d^{k}_{i-1},d^{k}_{i},d^{k}_{i+1})}{\Psi_0(d^{k}_{i-1},d^{k}_{i},d^{k}_{i+1})} = \frac{d^{k}_{i}\Psi_1(d^{k}_{i-1}/d^{k}_{i},1,d^{k}_{i+1}/d^{k}_{i})}{d^{k}_{i}\Psi_0(d^{k}_{i-1}/d^{k}_{i},1,d^{k}_{i+1}/d^{k}_{i})}
= \frac{\Psi_1(r^{k}_{i-1},1,R^{k}_{i})}{\Psi_0(r^{k}_{i-1},1,R^{k}_{i})}=:G_1(r^{k}_{i-1},R^{k}_{i}).$$\\
Since  $G_1$ is smooth when applied to positive data, 
applying the Mean Value Theorem we get  
\begin{align} \label{eq:MVT1} 
R^{k+1}_{2i}-1 &= G_1(r^{k}_{i-1},R^{k}_i) - G_1(1,1) =
\nabla G_1(\hat r^k_i,\hat R^k_i) \cdot \left ((r^{k}_{i-1},R^{k}_i) -
  (1,1)\right )
\end{align}
where $(\hat r^k_i,\hat R^k_i)$ belongs to the segment joining the
points $(1,1)$ and $(r^{k}_{i-1},R^{k}_i)$. Hence,
\begin{align}
|R^{k+1}_{2i}-1| 
&\leq \rho^{k} \max_{r,R\in K}  \| \nabla G_1(r,R) \|_1  , \label{eq:fitaR2i}
\end{align}
where $K$ is the ball of center $(1,1)$ and radius $\rho^k$. The same
type of argument can be used to obtain similar bounds 
for $|R^{k+1}_{2i+1}-1|$ and $|r^{k+1}_{2i+j}-1|$, $j=0,1$, which
proves the result.
\end{proof}

\newcommand{\lpinf}{l_\infty^+(\Z)}

\begin{rmk} \label{rmk:linf+}
Given $d \in \liZ$ with $d_i>0 (<0)$ $\forall i \in
  \Z$, we define
\begin{equation*} 
 \rho(d):=\sup_{i\in\Z}\left\{\left|r_{i}-1\right|,\left|R_i-1\right|\right\},
 \qquad  R_i:= \frac{d_{i+1}}{d_i}, \quad r_i = 1/R_i. 
\end{equation*}
Then, according to the previous lemma, $\rho: \lpinf \rightarrow
\R_+$, where
\[ \lpinf= \{ d \in \liZ: \, d_i \cdot d_{i+1} >0 \, \forall
  i \in \Z \text{ and } 
\sup_{i\in\Z}\left\{\left|\frac{d_{i}}{d_{i+1}}-1\right|,\left|\frac{d_{i+1}}{d_i}-1\right|\right\}< +\infty \} \]
and if $\rho^0=\rho(d^0)$ then $\rho^k$ in \eqref{eq:rho} satisfies
$\rho^k= \rho(d^k)$ with  $d^k= S_\epsilon^{(1)} d^{k-1}$, $k\geq 1$.

\end{rmk}



We will  show that,  at least under an appropriate restriction on
$\rho^0=\rho(d^0)$, the sequence
$(\rho^k)_{k\geq 0}$ converges to zero.
As a first attempt, we try to find $\eta$ and $\delta$ such that 
\begin{equation} \label{eq:rho:L1}
\rho^0<\delta \quad \Longrightarrow \quad \rho^{k+1} \leq \eta \rho^{k}, \qquad \forall k\geq 0.
\end{equation}
Obviously, if we could prove \eqref{eq:rho:L1} with $0<\eta<1$, we would  obtain
the desired result. However,
we shall see that we can only expect \eqref{eq:rho:L1} to 
hold for values of 
$\eta $ greater than 1. Nevertheless, we will be able to prove that given  $\eta\in (\frac34,1)$ there exists $\delta_\eta$ such that
\begin{equation} \label{eq:rho:L2}
\rho^0<\de \quad \Longrightarrow \quad \rho^{k+2} \leq \eta \rho^{k}, \qquad \forall k\geq 0,
\end{equation}
which will allow us to prove the required convergence result.

Let us start by analyzing \eqref{eq:rho:L1}.  From \eqref{eq:MVT1}, 
and taking into account that $||\nabla G_1(1,1)||_1 = \frac12$ (see the Appendix for details), given $\eta\in(\frac12,1)$ we can find $\delta_1>0$ such that
$$ \|(r,R)-(1,1)\|_\infty < \delta_1 \quad \Longrightarrow \quad \|
\nabla G_1(r,R)\|_1 \leq \eta.$$
Hence, from \eqref{eq:fitaR2i}, we get that if $\rho^k<\delta_1$
\begin{align*}
|R^{k+1}_{2i}-1| 
&\leq \rho^{k} \max_{r,R\in K}  \| \nabla G_1(r,R) \|_1 \leq \eta \,  \rho^k.
\end{align*}

On the other hand, the same type of arguments for $R^{k+1}_{2i+1}$
lead to the following:
$$R^{k+1}_{2i+1} = \frac{\Psi_0(d^{k}_{i},d^{k}_{i+1},d^{k}_{i+2})}{\Psi_1(d^{k}_{i-1},d^{k}_{i},d^{k}_{i+1})}
= \frac{\Psi_0\left (1,\frac{d_{i+1}^k}{d_i^k},\frac{d_{i+2}^k}{d_{i+1}^k}\frac{d_{i+1}^k}{d_i^k}\right )}{\Psi_1\left (\frac{d_{i-1}^k}{d_i^k},1,\frac{d_{i+1}^k}{d_i^k}\right )}
=:G_2(r^{k}_{i-1},R^{k}_{i},R^{k}_{i+1}).$$
It is shown in the appendix that $\|\nabla G_2(1,1,1)\|_1 = 1$, thus $|R^{k+1}_{2i+1}-1|$ cannot be bounded as before using any
$\eta< 1$ and we cannot expect \eqref{eq:rho:L1} to hold for any $\eta <1$. 
We then turn to analyze  \eqref{eq:rho:L2}. 
\begin{lem} \label{lem:contractivitat}
Given $\eta\in(\frac34,1)$, there exists  $\de>0$ 
such that if $\rho^0<\de$ then
$ \rho^{k+2} \leq \eta \rho^{k}, \qquad \forall k\geq 0. $
\end{lem}
\begin{proof}
We  shall use the same technique as before to examine
\[  |R^{k+2}_{4i+j}- 1| \quad \text{and} \quad  |r^{k+2}_{4i+j}-1|
\qquad j=0,1,2,3. \]

Let us denote the four rules that define  $(S_\epsilon^{[1]})^2$ as follows:
\[ (S_\epsilon^{[1]}S_\epsilon^{[1]} d)_{4i+j}  =:
\Psi^2_j(d_{i-2},d_{i-1},d_i,d_{i+1},d_{i+2}), \qquad j=0,1,2,3\]
where, without loss of generality, we have assumed that $\Psi^2_j: \R^5_+
\rightarrow \R_+$, $j=0,1,2,3$. The specific form of these (smooth)
functions can be found in the
Appendix, as well as the required
computations, which are lengthy but straightforward. In what follows
we give only a  sketch of the proof.

Since the functions $\Psi_j^2$ are 1-homogeneous and smooth, we can write 
\[ \begin{cases} 
 R^{k+2}_{4i+j} &= \displaystyle{\frac{\Psi^2_{j+1}(d_{i-2}^k,\cdots, d^k_{i+2})}{\Psi^2_j(d_{i-2}^k,\cdots, d^k_{i+2})}}=
G^2_j(r^{k}_{i-2},r^{k}_{i-1},R^{k}_{i},R^{k}_{i+1}), \\[10pt]
r^{k+2}_{4i+j}&=\displaystyle{\frac{1}{ R^{k+2}_{4i+j}}=
\frac{1}{G^2_j(r^{k}_{i-2},r^{k}_{i-1},R^{k}_{i},R^{k}_{i+1})}}, 
\end{cases} \quad j=0,1,2,3. \] 
It is easy to see that  $G^2_j(\uno_4)=1$, $\uno_4:=(1,1,1,1)$. In
addition, 
\[\nabla \frac{1}{G^2_j}(\uno_4)=
-\frac{ \nabla G_j^2(\uno_4)}{G_j^2(\uno_4)^2} = -\nabla G_j^2(\uno_4).\]

 In the appendix we
carry out all the necessary computations to  obtain the values  of 
$\|\nabla G^2_j(\uno_4) \|_1$, which are displayed in Table \ref{tab:gradient2}.
\begin{table}[!h]
\centering
\begin{tabular}{|c|c|c|c|c|} \hline
&&&&\\[-12pt]
$j$ & 0 & 1 & 2 & 3 \\\hline
&&&&\\[-10pt]
$\|\nabla G^2_{j}(\uno_4)\|_1$ &  \, 5/16 \, & \, 1/4 \, & \, 5/16 \, & \, 3/4 \, \\[2pt]\hline
\end{tabular}
\caption{The 1-norm of the ratio functions $G^2_j$ for $j=0,1,2,3$.} 
\label{tab:gradient2}
\end{table}

The functions $G^2_j:\R^4_+ \rightarrow \R_+ $ are smooth and satisfy 
$G^2_j(\uno_4)=1$,  $\|\nabla G^2_j(\uno_4) \|_1 \leq \| \nabla
G^2_3(\uno_4) \|_1 = \frac34$. Then using  the Mean Value Theorem as
before we know that,  given $\eta\in(\frac34,1)$ exists
$\delta$ such that if $\rho^{k}<\delta$ 
$$\rho^{k+2}= \max_{\iinZ} \max _{0\leq j \leq 3} \{ |R^{k+2}_{4i+j}-1|,
|r^{k+2}_{4i+j}-1| \} \leq \eta \rho^{k}.$$

To conclude the proof, we notice that given $\eta \in (3/4,1)$ and $k$ such that
$\rho^{k}<\delta$, we have that $\rho^{k+2}<\eta \rho^k
<\delta$. Thus, if $\rho^{0},\rho^{1}<\delta$, then
$\rho^{k}<\delta$ for all $k\geq 0$. 
Since we know that we can find $\delta^1>0$ and $\eta^1>1$ such that if
$\rho^0<\delta^1$, then  $\rho^1 \leq \eta^1 \rho^0$, taking
$\de:=\min \{ \delta/\eta^1,\delta^1\}$ we have that if
$\rho^0<\de \leq \delta^1$, then   $\rho^{1}< \eta^1 \rho^0 <
\eta^1 \de \leq \eta^1 \delta/\eta^1 =\delta$. 
Therefore if $\rho^0 < \de$, then $\rho^0, \rho^1 <
\delta$ and the result holds.
\end{proof}

Then, we can prove that $\rho^k$ decreases at least as fast as $\eta^{\frac{k}{2}}$.
\begin{prop} \label{prop:contractivitat}
Given $\eta\in(\frac34,1)$, there exists $ \de>0$ such that if
$\rho^0<\de$, 
\begin{equation*} 
\rho^{k+1} < \eta^\frac{k}{2} \max\{\rho^1,\rho^0\}, \qquad  \forall k\geq 0.
\end{equation*}
\end{prop}
\begin{proof}
Given $\eta\in(\frac34,1)$, we apply the previous Lemma and 
consider separately the cases of  $k$ even or odd. Then, for $\rho^0<\de$, 
\begin{align*}
k=2k' \quad \longrightarrow \quad \rho^{k+1} &= \rho^{2k'+1} \leq \eta
\rho^{2(k'-1)+1} \leq \cdots \leq \eta^{k'} \rho^{1}=\eta^{\frac{k}{2}}
  \rho^{1},\\
k=2k'+1 \quad \longrightarrow \quad \rho^{k+1} &= \rho^{2(k'+1)} \leq
\eta \rho^{2k'} \leq \cdots \leq \eta^{k'+1} \rho^{0}<   \eta^{\frac{k'+1}{2}} \rho^{0}< \eta^{k/2} \rho^0.
\end{align*}
\end{proof}

The next result shows that the growth of the divided finite
differences at each level or refinement also depends on $\rho^k$. Then
we can prove \eqref{eq:cond0}-(a).
\begin{lem} \label{lem:majorat} Assume that $d^k_i> 0$ (or $d^k_i<0$), $\forall
  \iinZ$. Then, 
\begin{equation*}
 (1 - \frac{\rho^k}{1+\sqrt{2}}  ) |d^{k}_i| \leq |d^{k+1}_{2i+j}| \leq (1 + \frac{\rho^k}{1+\sqrt{2}}) |d^{k}_i|, \qquad j=0,1.
 \end{equation*}
\end{lem}
\begin{proof}
As stated in (\ref{eq:S1expression}), for this type of data
\[d^{k+1}_{2i+j}= d^{k}_i\left (1+(-1)^j H(r^k_{i-1}, R^k_i)\right ), \qquad j=0,1.\]
It is very easy to check that $\forall x,y \geq 0$
$$ |H(x,y)| \leq  \frac{1}{2+2\sqrt{2}}|x-y| 
\leq \frac{1}{1+\sqrt{2}}\max\{|x-1|,|y-1|\}.$$
Thus $|H(r^{k}_{i-1},R^{k}_{i})|  \leq \frac{\rho^k}{1+\sqrt{2}}$
$\forall \iinZ$, and 
\[ 1-\frac{\rho^k}{1+\sqrt{2}} \leq 1-|H(r^{k}_{i-1},R^{k}_i)| \leq 
\frac{|d^{k+1}_{2i+j}|}{|d^{k}_i|}
\leq 1+|H(r^{k}_{i-1},R^{k}_i)| \leq 1+\frac{\rho^k}{1+\sqrt{2}}.\]
\end{proof}

\begin{prop} \label{lem:majorat2}
Let $d^0$ be a strictly positive (negative) sequence and $d^k=
S^{(1)}_\epsilon d^{k-1}$, $k \geq 1$. Given $\eta\in(\frac34,1)$,
$\exists \de>0$ such that if $\rho^0 < \de$ then the sequence
$(d^{k})_{k=0}^\infty$ is uniformly  bounded.
\end{prop}
\begin{proof}
By Lemma \ref{lem:majorat} 
\[\norma{d^{k+1}}\leq
(1+\frac{\rho^k}{1+\sqrt{2}} ) \norma{d^k}, \quad \rightarrow \quad 
\norma{d^k} \leq \norma{d^0} \prod_{l=0}^{k-1}
(1+\frac{\rho^l}{1+\sqrt{2}}).\] 
Applying Proposition \ref{prop:contractivitat}, there exists
$\de >0$ such that if $\rho^0 <\delta_\eta$, 
$$ \norma{d^k} \leq \norma{d^0} \prod_{l=0}^{k-1} (1+c \eta^{l/2} ),
\qquad  c = \max\{\rho^1,\rho^0 \} \frac{\eta^{-1/2}}{1+\sqrt{2}}. $$
Observe that
$$ \prod_{l=0}^{k-1} (1+c \eta^{l/2} ) = \exp(\log(\prod_{l=0}^{k-1} (1+c \eta^{l/2} ))) = \exp(\sum_{l=0}^{k-1} \log(1+c \eta^{l/2} )) \leq \exp(\sum_{l=0}^{k-1} c \eta^{l/2} ),$$
since $\log(1+t) \leq t$,  $\forall t\geq 0$. Then
$ \norma{d^k} \leq \norma{d^0}\exp(\frac{c}{1-\sqrt{\eta}} ),$
which proves the result.
\end{proof}

\begin{teo} \label{teo:smooth}
Let $d^0$ be a strictly positive (negative) sequence and $d^k=
S^{(1)}_\epsilon d^{k-1}$, $k \geq 1$. There exists $\bar \delta >0$ such
that if $\rho^0<\bar \delta$, then 
$S_\epsilon^{(1)}d^0$ converges to a $\cC^{\alpha-}$ function, with
$\alpha= 1-\frac12 log_2 3$.
\end{teo}
\begin{proof}
As we observed previously, the  piecewise linear functions
$ \mathbb{P}^{k}(t) $ in \eqref{eq:Pkdef} satisfy that 
\[\norma{ \mathbb{P}^{k+1} -  \mathbb{P}^{k}} \leq \norma{\nabla
  d^{k+1}}.\]
Since $ |\nabla d^{k}_j| = |d^{k}_{j+1}-d^{k}_{j}|
= |d^{k}_{j}|\left |\frac{d^{k}_{j+1}}{d^{k}_{j}}-1\right |, $
by  Lemma \ref{lem:majorat2} and  Proposition \ref{prop:contractivitat}, given $\eta\in(\frac34,1)$, $\exists
\de>0$ and $C>0$ such that if $\rho^0 < \de$ then  
\[\norma{ \mathbb{P}^{k+1} -  \mathbb{P}^{k}} \leq \norma{\nabla
  d^{k+1}}\leq \norma{d^{k+1}} \rho^{k+1} \leq C \eta^{\frac{k}{2}}.\]
Then, by slightly modifying the proof of Corollary 3.3 in
\cite{Dyn92} we get that $S^{(1)}_\epsilon d^0$ is 
$\cC^{-\log_2(\sqrt{\eta})}$ smooth. 

Let us define  $\bar \delta:= \sup_{\frac34 < \eta < 1}\de$. If
$\rho^0< \bar \delta$, then $\rho^0 <\delta_{\eta_0}$ for some $\eta_0 \in
(\frac34,1)$ and, by the properties of $\delta_{\eta_0}$, we also have that
$\rho^k \leq C \eta_0^{k/2}$ and, as a consequence,  $S^{(1)}d^0$ is
$\cC^{-\log_2(\sqrt{\eta_0})}$ smooth.
  Let us consider now 
any other $\eta \in (\frac34,1)$, and its 'associated' $\de$.
Let $k_0$ be s.t. $\eta_0^{k/2} < \de/C$, $\forall k \geq k_0$. Then $\rho^k
< \de$, $\forall k \geq k_0$ and, by the same arguments used
throughout this section, $\rho^k \leq C' \eta^{\frac{k}{2}}$,  $\forall
k \geq k_0$, $C'>0$. But this decay rate on the $\rho^k$ (for $k$ large
enough) also implies  that the smoothness of the limit function
is $\cC^{-\log_2(\sqrt{\eta})}$. Since $\eta$ is arbitrary
in $(\frac34,1)$, we  conclude that
the limit functions are at least $\cC^{\alpha-}$, with
$\alpha=-\frac12\log_2(\frac34)= 1-\frac12\log_2(3) \cong 0.2 $.
\end{proof}

According to the previous theorem,  the smoothness of $S_\epsilon$
for strictly monotone data is at least $C^{1+\alpha -}$, as long as
the initial data satisfies the  'technical' additional condition
$\rho(d^0) <\bar \delta$. Notice that
\[ \rho(d) < \bar \delta \quad \equiv \quad \max_i
\{\frac{|d_{i+1}-d_i|}{|d_i|}, \frac{|d_{i+1}-d_i|}{|d_{i+1}|} \} < \bar \delta, \]
hence, it is quite straightforward to see that the required technical condition
may, in fact,  be easily achieved for 
smooth (strictly increasing)  data. Indeed, if $f=F|_{2^{-k} \Z}$,
then $d_i=2^k \nabla f_i= F'(i h)+ O(h)$, with $h=2^{-k}$, and
$$\frac{|d_{i+1}-d_i|}{|d_i|} \approx 
h \frac{|F''(ih+\xi)| + O(h)}{|F'(ih)| + O(h)}, \quad \xi\in[0,h].$$
Hence, for strictly monotone smooth initial data such that $|F'|>\theta>0$,
we get that $\rho(d)$ is  $O(h)$ and there should be no problem in
adjusting $h$ (the initial sampling) in order to fulfill the required
condition.

In addition, taking into account the proof of Lemma
\ref{lem:contractivitat}, it is possible to  give an estimate of the value
of $\bar \delta$ with the aid of Wolfram Mathematica by checking what
is the largest value of $\delta>0$ satisfying
\[  \| \nabla G^2_j ( \bar x ) \|_1 < 1, \quad 
\| \nabla \frac{1}{G^2_j} ( \bar x ) \|_1 < 1, \quad \forall \, \bar x :||\bar x -
\uno_4||_\infty <\delta .\]
According to our computations,  we  estimate that $\bar \delta \approx 0.18$.

\section{Stability and Approximation order \label{sec:approx}}

Stability and  approximation order  are also important
properties of a subdivision scheme. Both concepts are enclosed below
for completeness. 
\begin{defi}{\em (Lipschitz) Stability.} \label{defi:stability}
We say that a convergent subdivision scheme is \emph{stable} if
$$ \norma{\Sinf f - \Sinf g}\leq C\norma{f-g}, \qquad f,g\in\liZ.$$
\end{defi}
\begin{defi}{\em Approximation order.}
A convergent subdivision scheme has \emph{approximation order $r$} if
for any sufficiently smooth function, $F$, there exists $h_0$ such that
$$\norma{F - (\Sinf f^0)(\bullet/h)}\leq Ch^r, \qquad f^0=F|_{h\Z}, \qquad
\forall h\leq  h_0.$$
\end{defi}
The approximation order measures the {\em approximation capabilities}
of the subdivision process, that is, the ability to ensure that smooth
behavior is adequately represented.
Since the explicit expression of $\Sinf$ is usually unknown, the
approximation capabilities of a subdivision scheme are often analyzed
by considering instead the approximation order \emph{after one step}.
\begin{defi}{\em Approximation order after one step.}
A subdivision scheme has \emph{approximation order $r$ after one step}
if for any sufficiently smooth function, $F$, there exists $h_0$ such that
$$\norma{ F|_{\frac{h}{2}\Z} - S f } \leq Ch^r, \qquad f=F|_{h\Z}, \qquad \forall h\leq  h_0.$$
\end{defi}

It is well known  that the approximation order of a
subdivision scheme {\em   after one step} determine the order of
approximation of the scheme, provided the scheme is stable (Theorem 2.4.10 of \cite{Kuijt98}). The order of approximation and the stability of
nonlinear schemes are often studied together \cite{ADLT06,ADS16,DL-US17}.


For linear stationary
subdivision schemes, Lipschitz stability is a consequence of
convergence, but this is not the case for  nonlinear subdivision
\cite{ADL11,CDM03,HO10}. Some theory was developed and successfully applied on
several instances \cite{ADL11,ADS16,CDM03,DL08,DL-US17,Kuijt98}.

We have already observed that the nonlinear schemes $S_\epsilon$ in
\eqref{eq:seps-def} are not stable for general data (see Figure
\ref{fig:rep_cercle_min}).  
However, as in section \ref{sec:smoothness}, we shall
be able to prove stability for a conveniently restricted 
class of strictly monotone data. For such data, the order of
approximation can be obtained by the usual, Taylor-like, one-step
approximation results \cite{DL-US17}.

\subsection{Approximation order}

\begin{teo} \label{teo:approx}
Let $F$ be a smooth function with $|F'|>\theta>0$, and let
$f=(F(ih))_{i\in\Z}$. Then, for any $\epsilon \in [0, \sqrt{2}]$ we have
\[|| F|_{\frac{h}{2}\Z} - S_\epsilon f||_\infty \leq C h^4 . \]
\end{teo}
\begin{proof}
Clearly $f^0=(F(ih))_{i\in\Z}$ is a strictly monotone sequence. Since
$(S_\epsilon f^0)_{2i}=f^0_i=F(ih)=F(2 i h/2)$, we only need to measure the
distance between $(S_\epsilon f)_{2i+1}$ and $F(\xi_i)$, $\xi_i = (2i+1)h/2 = (i+1/2)h$. By taking a formal
Taylor series expansion we find 
$$(S_\epsilon f)_{2i+1}= F(\xi_i)+\frac{3 h^4}{128} \left(\frac{F^{(3)}(\xi_i) F''(\xi_i)}{F'(\xi_i)}-F^{(4)}(\xi_i)\right)+O\left(h^5\right).$$
\end{proof}

\subsection{Stability} \label{sec:stable}

In \cite{ADL11,ADS16,DL-US17},  stability is proved using a result
similar to Theorem  \ref{teo:conv1} (see for instance Theorem 1 of
\cite{DL-US17}), which requires that the scheme  is of the form
\eqref{eq:nonlinearSS}. 
\begin{teo} \label{teo:stable}
 Let $S$ be of the form (\ref{eq:nonlinearSS}).
$S$ is stable provided that
\begin{align*}
\text{\textbf{S1.}} & \qquad \exists M > 0 :&  ||\cF(f)-\cF(g) ||_{\infty} \leq M ||f-g||_{\infty} &&   \forall f,g \in  l^{\infty}(\mathbb{Z})\\[3pt]
\text{\textbf{S2.}} & \qquad \exists  L > 0, \, \,  0<\eta < 1 :&
\norma{ (S^{[1]})^L f - (S^{[1]})^L g } \leq \eta  \norma{f-g }
&& \forall f,g \in
l^{\infty}(\mathbb{Z})
\end{align*}
\end{teo}

Notice that \textbf{S1} and \textbf{S2} are Lipschitz-type conditions
on  $\cF$ and $(S^{[1]})^L$. We have already observed that $S_\epsilon$ is not stable (see Figure
\ref{fig:rep_cercle_min}). 
 The reason behind the lack of stability can be traced back to the fact that
 $\Gamma_\epsilon$ is not a continuous function, so that \textbf{S1} and
 \textbf{S2} cannot be fulfilled, in general.

In section \ref{sec:smoothness} we have seen  that if the initial
data is a strictly monotone sequence and $\epsilon\in [0,\sqrt{2}]$, 
the subdivision rules of $S^{(1)}_\epsilon$ are smooth, positive,
functions. In this case, the two conditions in 
Theorem \ref{teo:stable} could be fulfilled, hence, for the rest of
the section we shall (silently) assume that $\epsilon\in [0,\sqrt{2}]$,
and restrict our attention to strictly monotone data. 

In \cite{ADS16,DL-US17,HO10}, the authors use the theory of {\em
  Generalized Jacobians} to prove  condition \textbf{S2} for nonlinear schemes
defined by piecewise smooth subdivision rules. The main argument used
in these references derives from the following inequality 
\begin{equation} \label{eq:stable-L}
\norma{ (S^{[1]})^L f - (S^{[1]})^L g } \leq \norma{f-g} \sup_{t\in[0,1]} \|DS^{[1]}(\tau^{L-1}(t)) \cdots DS^{[1]}(\tau^{1}(t)) DS^{[1]}(\tau^{0}(t))\|_\infty,
\end{equation}
where $DS^{[1]}$ is the (generalized) Jacobian of $S^{[1]}$ (see
\cite{HO10} or the appendix in \cite{DL-US17} for details) and 
\begin{equation} \label{eq:tau}
\tau^0(t) = (1-t)f + t g, \qquad \tau^k(t) = S^{[1]}(\tau^{k-1}(t)).
\end{equation}
In our case, the (smooth)  subdivision rules of $S^{[1]}_\epsilon=
\frac12 S^{(1)}_\epsilon$ are $\frac12 \Psi_0, \frac12 \Psi_1$, with
$\Psi_j$ defined in \eqref{eq:S1expression} and $DS_\epsilon^{[1]}$ is
the bi-infinite matrix with the following non-zero entries 
(we use Matlab notation, as in \cite{ADS16,DL-US17}))
\begin{equation} \label{eq:generalized_jacobian}
(DS_\epsilon^{[1]})_{[2i+j,i-1:i+2]} = \frac12 \nabla
\Psi_j(f_{i-1},f_{i},f_{i+1}), \qquad j=0,1 .
\end{equation}
 To check \eqref{eq:stable-L}, we need the following preliminary results.

 \begin{lem} \label{lema:U_close}
Given $\delta>0$, and $f,g \in \liZ$ such that $ f_i, g_i>0(<0)$
$\forall i\in\Z$, we have that
\begin{equation} \label{eq:tau0}
\rho(f),\rho(g)<\delta \quad \rightarrow \quad 
\rho((1-t) f + t g)<\delta \quad \forall t\in[0,1],
\end{equation}
where $\rho$ defined in Remark \ref{rmk:linf+}. 
\end{lem}
\begin{proof}
Notice first that for any $u \in \liZ$ 
$$ \rho(u)<\delta \quad \equiv \quad |u_{i+1}-u_i| \leq \delta |u_j|,
\qquad \forall \iinZ , \, j=i,i+1.$$

To prove the result, we shall check that  $ \forall \iinZ$   and $t\in[0,1]$,
$$ |((1-t) f_{i+1} + t g_{i+1})-((1-t) f_i +t g_i)| <  \delta |(1-t)
f_j+ t g_j|,\qquad j=i,i+1.$$
The cases $t=0$ and $t=1$ are trivial. 
 For, $t\in (0,1)$,  since $(1-t) >0$ we can write 
\begin{align*}
|((1-t) f_{i+1} + t g_{i+1})-((1-t) f_i +t g_i)| &\leq (1-t)|
f_{i+1}-f_i| +  t |g_{i+1}-g_i|   \\
&\leq \delta  \left ( (1-t)|f_j| + t |g_j| \right )=
\delta |(1-t) f_j + t g_j |, \quad  j=i,i+1 , 
\end{align*}
where we have used that $| f_l +  g_l| = | f_l| + |
 g_l|$ $\forall l \in \Z$, since all components have the same sign.
\end{proof}

\begin{prop} \label{prop:gradiente1}
Let us consider $\eta\in(\frac58,1)$ and $\delta\in(0,\bar \delta)$,
where $\bar \delta$ is defined in Theorem \ref{teo:smooth}. Then
$\exists \, K_\delta$ such that $\forall f^0,g^0$ satisfying $ f_i^0,
g_i^0>0(<0)$ and  $\rho(f^0),\rho(g^0)< \delta$
$$ \sup_{t\in[0,1]} \|D S_\epsilon^{[1]} (\tau^k(t)) \|_\infty \leq
\eta, \qquad \forall k\geq K_\delta,$$
with $\tau^k$ in \eqref{eq:tau}.
\end{prop} 
\begin{proof}
From \eqref{eq:generalized_jacobian}, we have 
\begin{equation} \label{eq:normDS}
\|D S_\epsilon^{[1]} (\tau^k(t)) \|_\infty = \sup_{i\in\Z, \, j=0,1}
\| \frac12 \nabla\Psi_j(\tau^k_{i-1}(t),\tau^k_{i}(t),\tau^k_{i+1}(t))
\|_1, \qquad \forall t\in[0,1]
\end{equation}
where 
$\tau^k_l(t)=(\tau^k(t))_l$, $l \in \Z$. 
Since  
$\nabla \Psi_j$ is 0-homogeneous and $\tau^k_i(t)\neq 0$, $\forall t
\in [0,1]$, we can write
\begin{equation} \label{eq:prop:gradiente1:1}
 \nabla \Psi_j (\tau^k_{i-1}(t),\tau^k_{i}(t),\tau^k_{i+1}(t)) = \nabla \Psi_j \left (\frac{\tau^k_{i-1}(t)}{\tau^k_{i}(t)},1,\frac{\tau^k_{i+1}(t)}{\tau^k_{i}(t)}\right ).
 \end{equation}
It is easy to check (see Appendix) that $\frac12 \|\nabla \Psi_0 (1,1,1)
\|_1 = \frac12\|\nabla \Psi_1 (1,1,1) \|_1 = \frac58$. Hence, given
$\eta \in(\frac58,1)$ (and  using 
the Mean Value Theorem as before)   $\exists \lambda_\eta>0$ such that
\begin{equation} \label{eq:prop:gradiente1:2}
\rho(\tau^k(t)) 
<\lambda_\eta \quad \Rightarrow \quad
\left \|  \frac12 \nabla  \Psi_j \left
    (\frac{\tau^k_{i-1}(t)}{\tau^k_{i}(t)},1,\frac{\tau^k_{i+1}(t)}{\tau^k_{i}(t)}\right
  ) \right \|_1 \leq \eta,  \, \, \,   j=0,1 \quad   \Rightarrow \quad
\|D S_\epsilon^{[1]} (\tau^k(t))\|_\infty \leq \eta .
\end{equation}

 Consider $\delta< \bar \delta$ and  recall 
that  $\bar \delta =
\sup_{\frac34<\eta'<1} \delta_{\eta'}$,  where $\delta_{\eta'}$ is
given in the Proposition \ref{prop:contractivitat}.

   If    $\rho(f^0),\rho(g^0)<\delta$, by
   Lemma \ref{lema:U_close}, $\rho(\tau^0(t))<\delta$   $\forall
   t\in[0,1]$.  
Then, since $\delta < \bar \delta$, 
   there exists $\eta'\in(\frac34,1)$ such that
 $$\rho(\tau^{k+1}(t))<(\eta')^{k/2}
 \max\{\rho(\tau^1(t)),\rho(\tau^0(t))\}  \leq C \delta (\eta')^{k/2},\quad  \forall
 t \in [0,1].$$
since, by using the arguments in   Lemma 17, we can easily get that $\rho(\tau^1(t))\leq c \rho(\tau^0(t))$ $\forall t
\in [0,1]$, with $c$ independent of $t \in [0,1]$. Thus, there exists
$K=K_\delta$ such that $\rho(\tau^k(t)) < 
\lambda_\eta$ for all $k\geq K_\delta$ and $\forall t \in
[0,1]$. Hence  the result follows  from \eqref{eq:prop:gradiente1:2}. 
\end{proof}

 Using these results, a partial stability result can be
  stated: When applied to strictly monotone data,  $S_\epsilon$ is
  {\em stable} with respect to strictly monotone perturbations, as
  long as the initial data $f^0$ and the perturbation $g^0$ satisfy
  a  {\em technical} condition on the sizes of $\rho(\nabla f^0)$,
  $\rho(\nabla g^0)$. 

From \eqref{eq:Sepsform10}-\eqref{eq:Sepsform10F}, and the results in
section \ref{sec:mono},  we know that $\Gamma_\epsilon^{[1]}$ is
Lipschitz for this kind of data. On the other hand,  assuming that
$\rho(\nabla f^0),\rho(\nabla g^0) < \delta <\bar \delta$ and using
Proposition  \ref{prop:gradiente1} we can  find 
$L_\delta>0$ such that
$$ \sup_{t \in [0,1]} \|DS_\epsilon^{[1]} (\tau^{L_\delta-1}(t))\|_\infty 
\|DS^{[1]}_\epsilon(\tau^{L_\delta-2}(t)) \|_\infty \cdots 
\|DS^{[1]}_\epsilon (\tau^{0}(t))\|_\infty < 1,$$ 
where $ \tau^0(t)=(1-t)\nabla f^0+t \nabla g^0, \,
\tau^k(t)=S^{[1]}_\epsilon \tau^{k-1}(t),\,  k \geq 1$. 
This is sufficient to ensure stability for this (restricted) class of
initial data (see \cite{ADS16}).

\section{Numerical experiments} \label{sec:numeric}
In the present section we present  several numerical experiments that
illustrate  the theoretical results obtained in this paper.
Throughout this section, we shall always consider $S_\epsilon$ with
$\epsilon=1$, which belongs to the range of values for which we can
ensure that  the
scheme is convergent,  reproduces  trigonometric functions (with
$|\gamma|\leq \frac{2\pi}{3}$), hyperbolic functions and second order
polynomials. In addition, in strictly monotone regions, it is $\cC^1$,
stable (under strictly monotone perturbations) and, hence, it has approximation
order 4.

\subsection{Reproduction properties} 
As stated in the introduction, the exact reproduction of specific
families of functions is a valuable asset for a subdivision
process. By applying a convergent interpolatory subdivision scheme to each one of the  coordinates of  an initial data set $f^0=(x^0_i,y^0_i)_{i\in\Z}$, one readily obtains a continuous curve $(x(t),y(t))=(\Sinf x^0,\Sinf y^0)=\Sinf f^0$ that interpolates the initial data set. 
Here, we will check the exact reproduction property of our
scheme when applied to different conic sections.

 In Figure \ref{fig:combination}-left we consider an anthropomorphic shape
 formed by an ellipse, two hyperbolas and a parabola. We take 7 points
 on each one of the conic sections, i.e. 28 
 points in total, that are repeated periodically to form
 $f^0$. $S^\infty_\epsilon f^0$ is shown in the center plot of
 Figure \ref{fig:combination} (after 7 applications of the subdivision
 process). The plot shows that  each conic section is correctly
 reproduced. In Table \ref{tab:repro} we show the errors between
 $S^\infty_\epsilon f^0$ and the value of the conic section at each one of the
 points  marked with an $*$ in  the left plot.
 The table shows that the error is of the order of  machine precision
 in each case, confirming the exact reproduction properties of the
 scheme.  We remark here that the scheme is able to {\em exactly reproduce}
 each one of the conic sections without any knowledge of the type of
 conic to which it is being applied. 
 \begin{figure}[ht]
\centering
\includegraphics[height=0.2\textheight,  
clip]{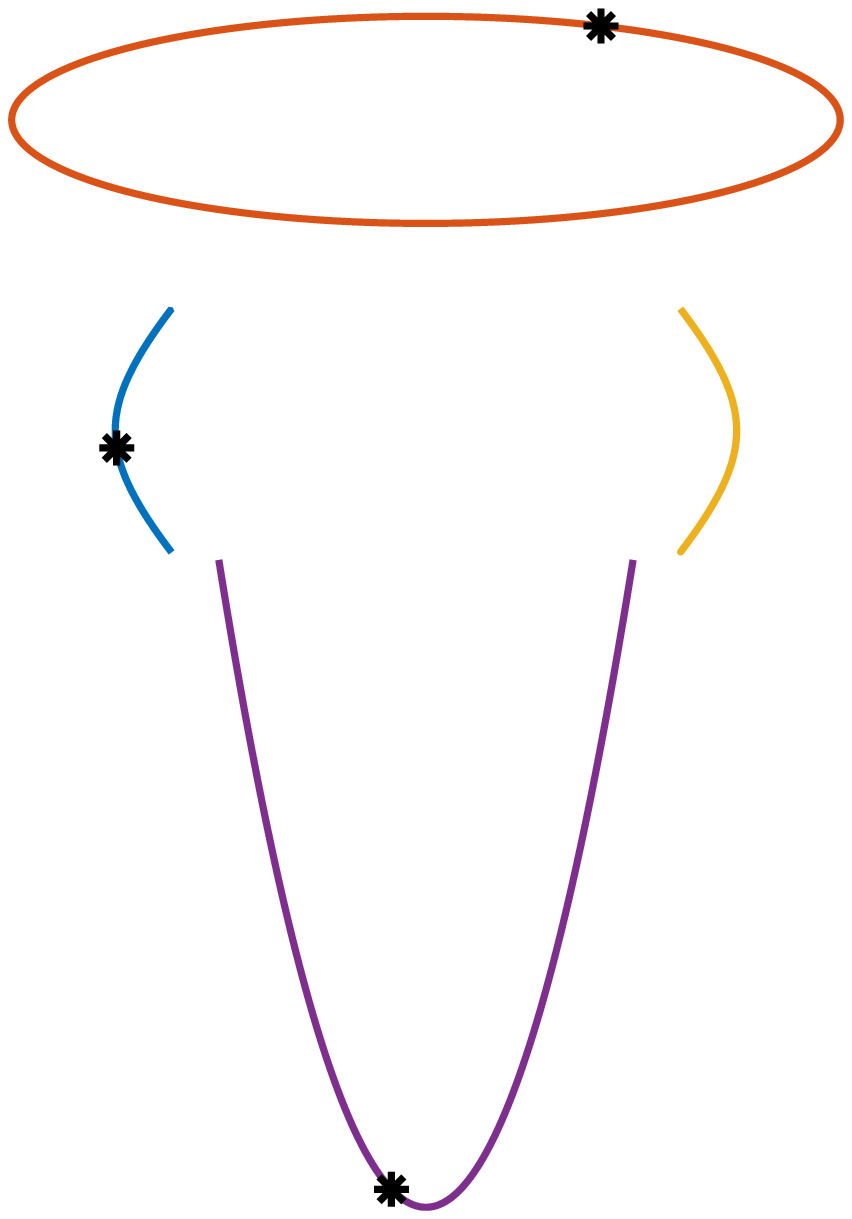} \ \ 
\includegraphics[height=0.2\textheight, clip, trim={0 0 6.5cm
  0}]{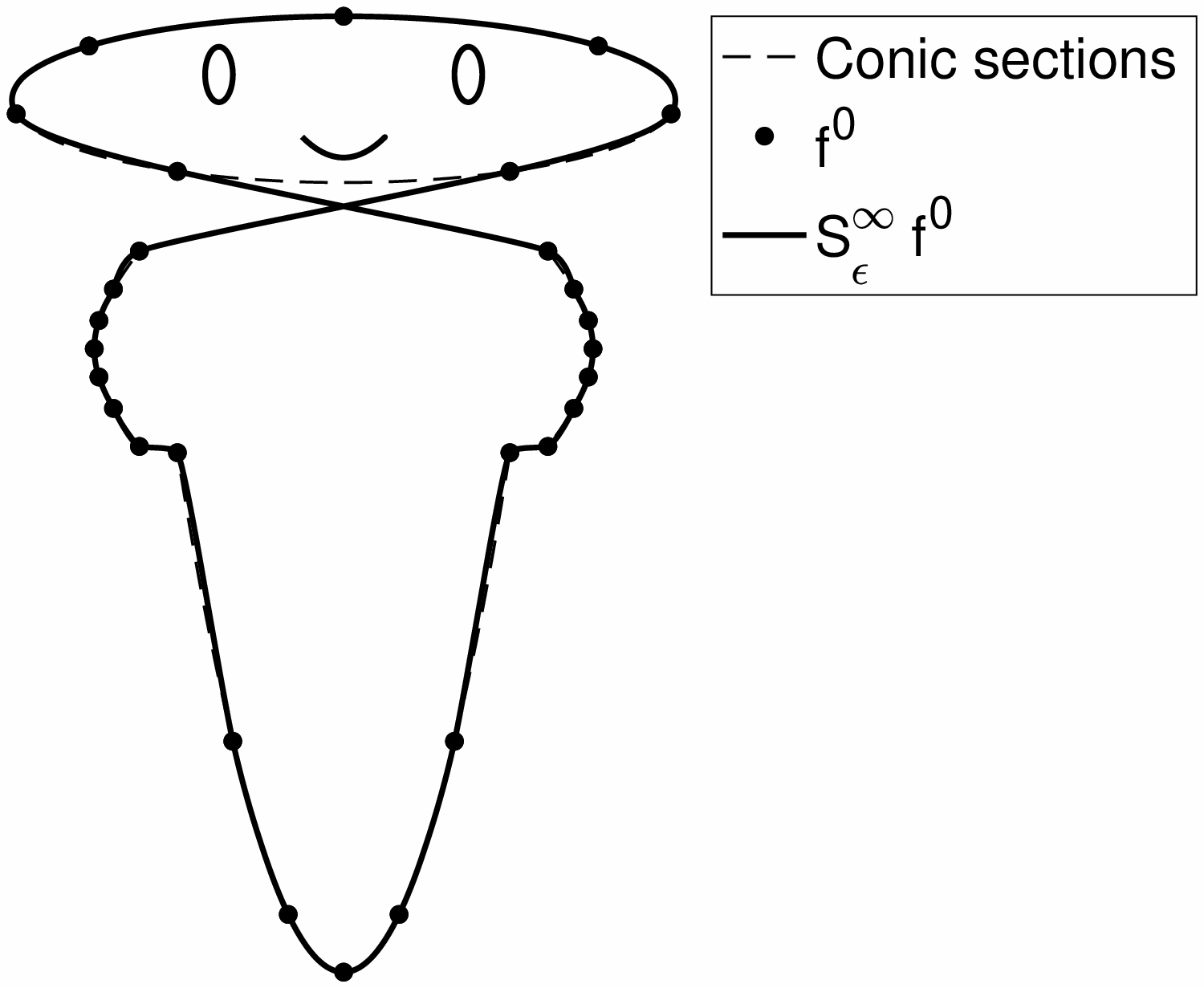} \ \ 
\includegraphics[height=0.2\textheight, clip, trim={0 0 6.5cm 0}]{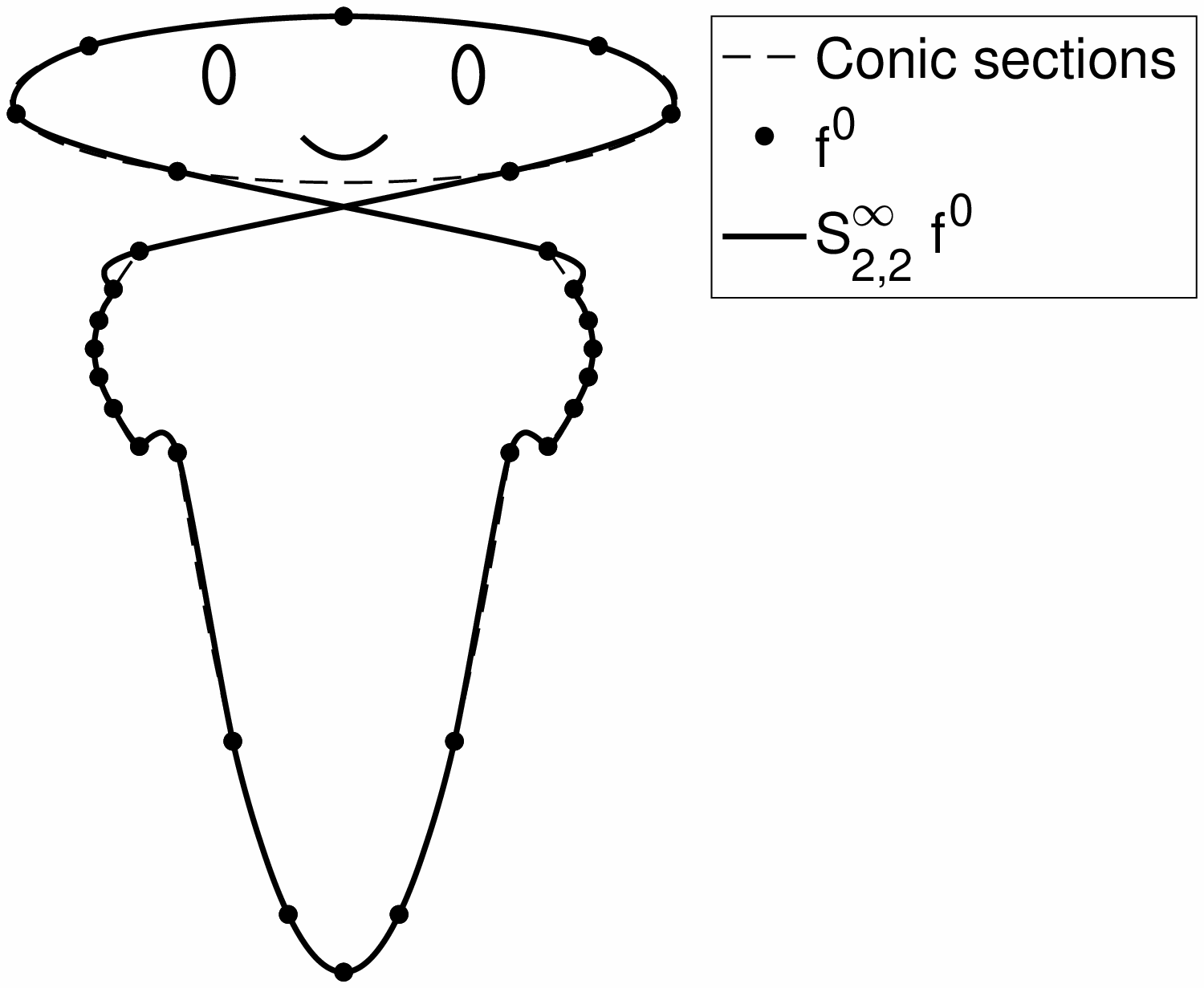}
\caption{Left plot: Anthropomorphic shape composed of one ellipse, two
  hyperbolas and one parabola. The marked points refer to 
Table \ref{tab:repro}. Center plot: $f^0$, black dots. $S^\infty_\epsilon
  f^0$, solid line ($\epsilon=1$). Right plot: $f^0$, black
  dots. $T_{2,2}^\infty f^0$, solid line. The 'exact' conic sections
  are represented with a dashed line in the center and right plots.  \label{fig:combination}}
\end{figure} 
\begin{table}[!h] 
\centering
\begin{tabular}{|c||c|c|c|}
\hline
point/scheme				& Ellipse
&	Hyperbola		&	Parabola

\\\hline \hline
$S_\epsilon$	&	6.6613e-16	&	4.4755e-16	&
2.2204e-16	\\\hline 
$T_{2,2}$		&	2.4467e-02	&	3.0012e-04	&	9.1551e-16	\\\hline
\end{tabular}
\caption{\label{tab:repro}  Error between  $S_\epsilon^\infty f^0$ and
  $T_{2,2}^{\infty}f^0$ and the correct value of each one of the points marked in Figure \ref{fig:combination} left.}
\end{table}

In addition, a non-oscillatory shape is
 obtained  in the {\em transition zones} between two
 conic sections. This behavior  is a
 distinctive feature of our scheme, when compared with its linear
 counterparts.  For the sake of comparison, we also show
 $T_{2,2}^\infty f^0$ in the right plot 
of Figure \ref{fig:combination}. In this case,
only the parabola is exactly reproduced, as confirmed by  Table \ref{tab:repro}. The  oscillatory
behavior in the transition zones can be clearly appreciated.

\subsection{Monotonicity Preservation.  Smoothness of limit functions} \label{sec:numerical:smooth}
We have proven in section \ref{sec:mono} that monotone data is
preserved by $S_\epsilon$ when $ \epsilon \in [0, \sqrt{2}]$. If the data is
strictly monotone, then this feature is also preserved. To check
numerically this property, we consider as a 
test case the monotone data of Table 3 in \cite{ADS16}: 
\begin{equation} \label{eq:monotone_data1}
f^0=(10,10,10,10,10,10.5,10.5,10.5,10.5,15,50,50,50,50,60,85,85,85,85).
\end{equation}
The limit (monotone) function  $S^\infty_\epsilon f^0$ is displayed 
in the left plot of  Figure \ref{fig:monotone} (solid line). For the
sake of comparison, $T_{2,2,}^\infty f^0$ is also shown (dotted
line). The different behavior between both limit functions can be
clearly appreciated in the right plot, which shows a zoom  of the flat
region (between jumps) marked with a rectangle on the left plot. 

\begin{figure}[ht]
\centering
\includegraphics[clip,width = 0.35\textwidth]{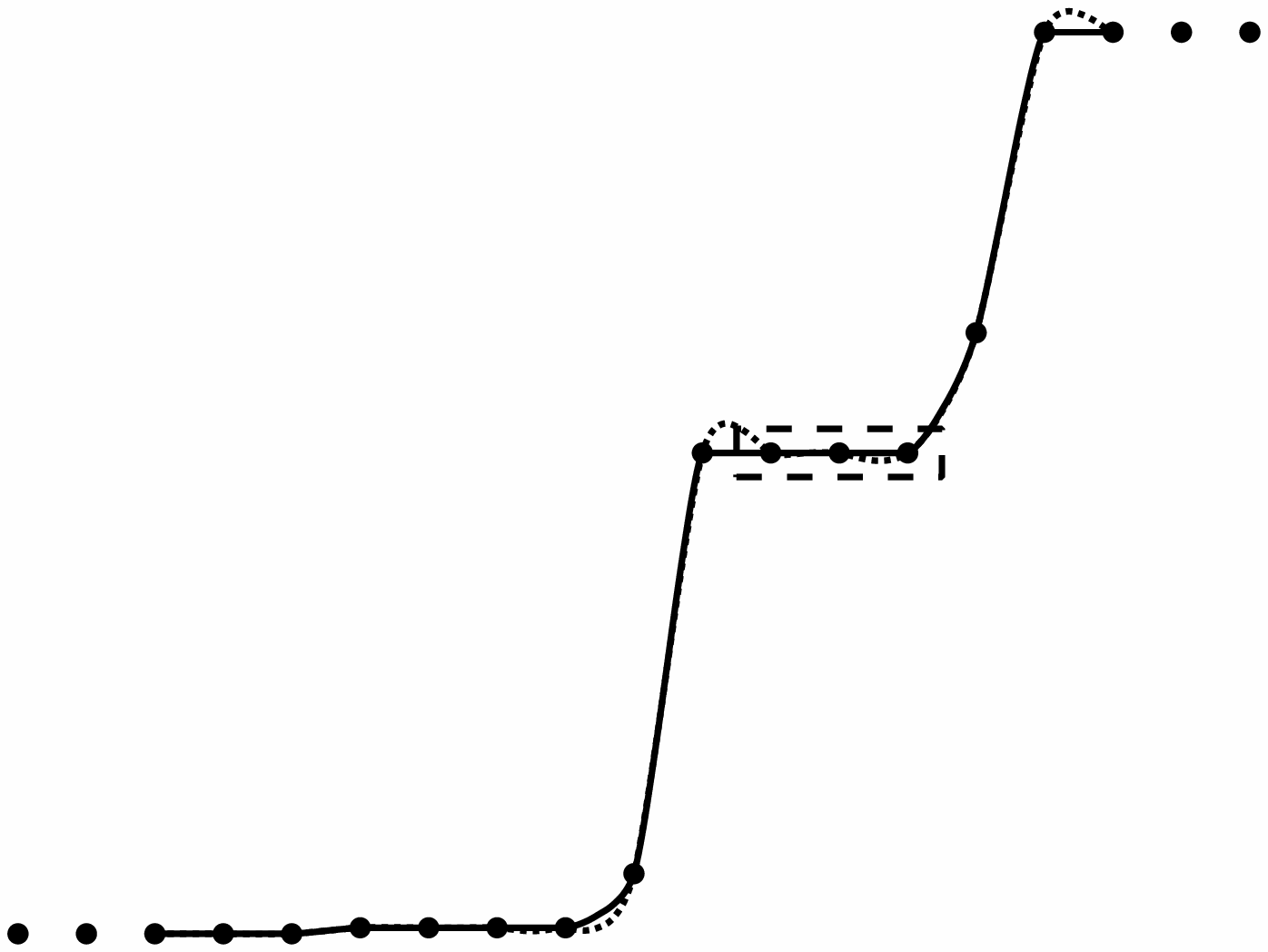} \hspace{20pt}
\includegraphics[clip,width = 0.45\textwidth]{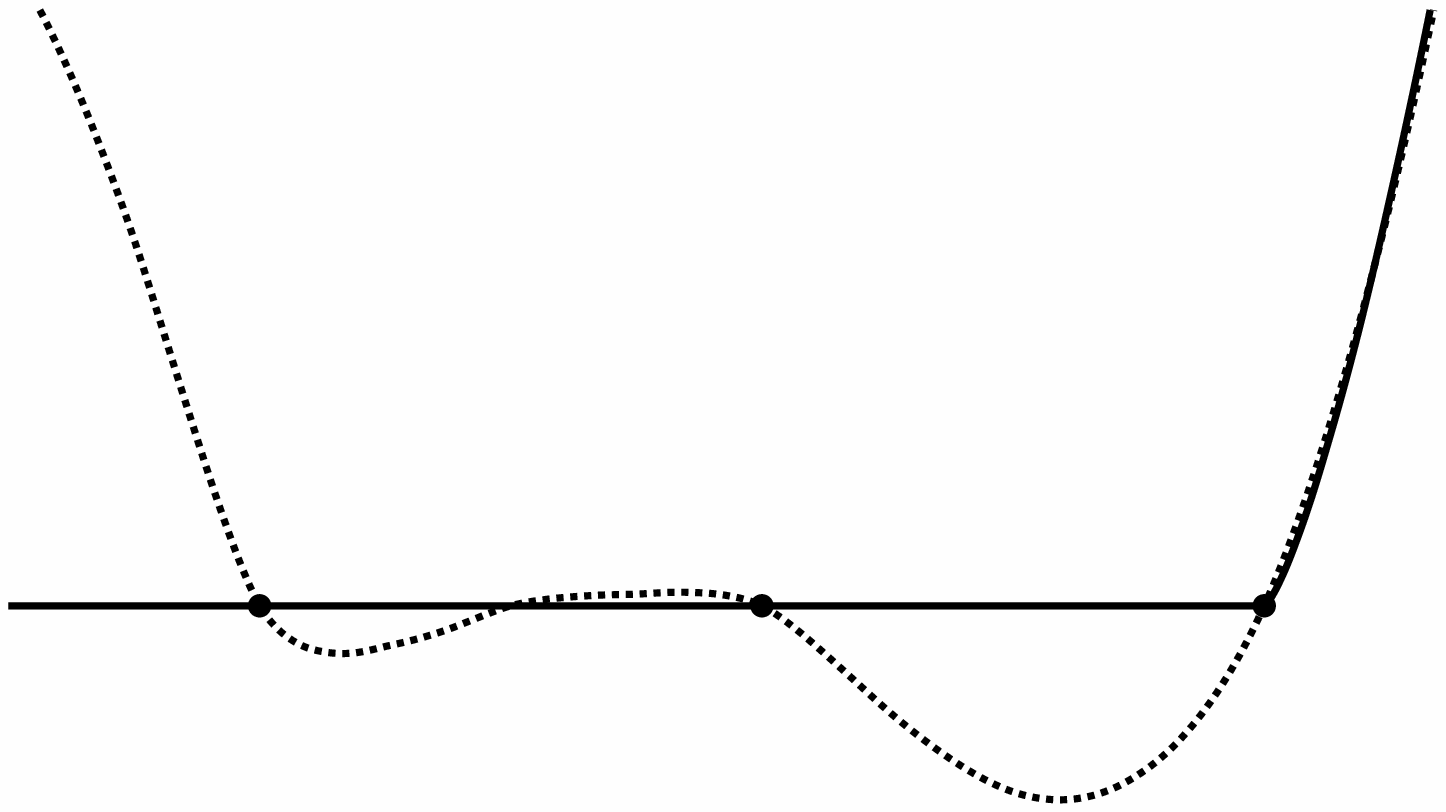}
\caption{Left plot: $T_{2,2}^\infty f^0$ (dotted line) and
  $S_\epsilon^\infty f^0$ (solid line) with  $f^0$ in
  \eqref{eq:monotone_data1}. Right plot: zoom of the area  marked
   with a dashed rectangle on the left plot.
\label{fig:monotone}}
\end{figure}
 
\begin{figure}
\centering
\includegraphics[clip,width = 0.35\textwidth]{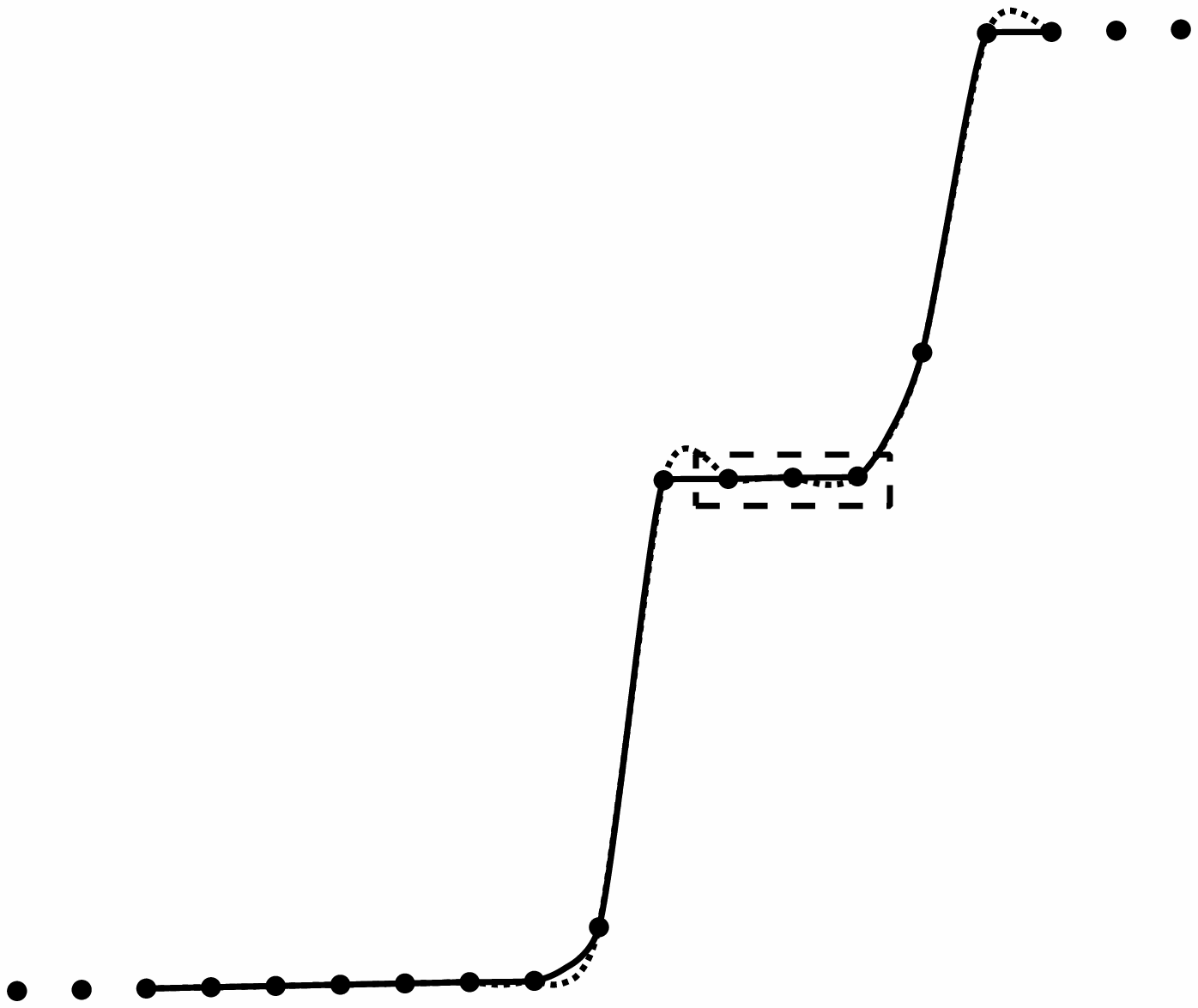} \hspace{20pt}
\includegraphics[clip,width = 0.45\textwidth]{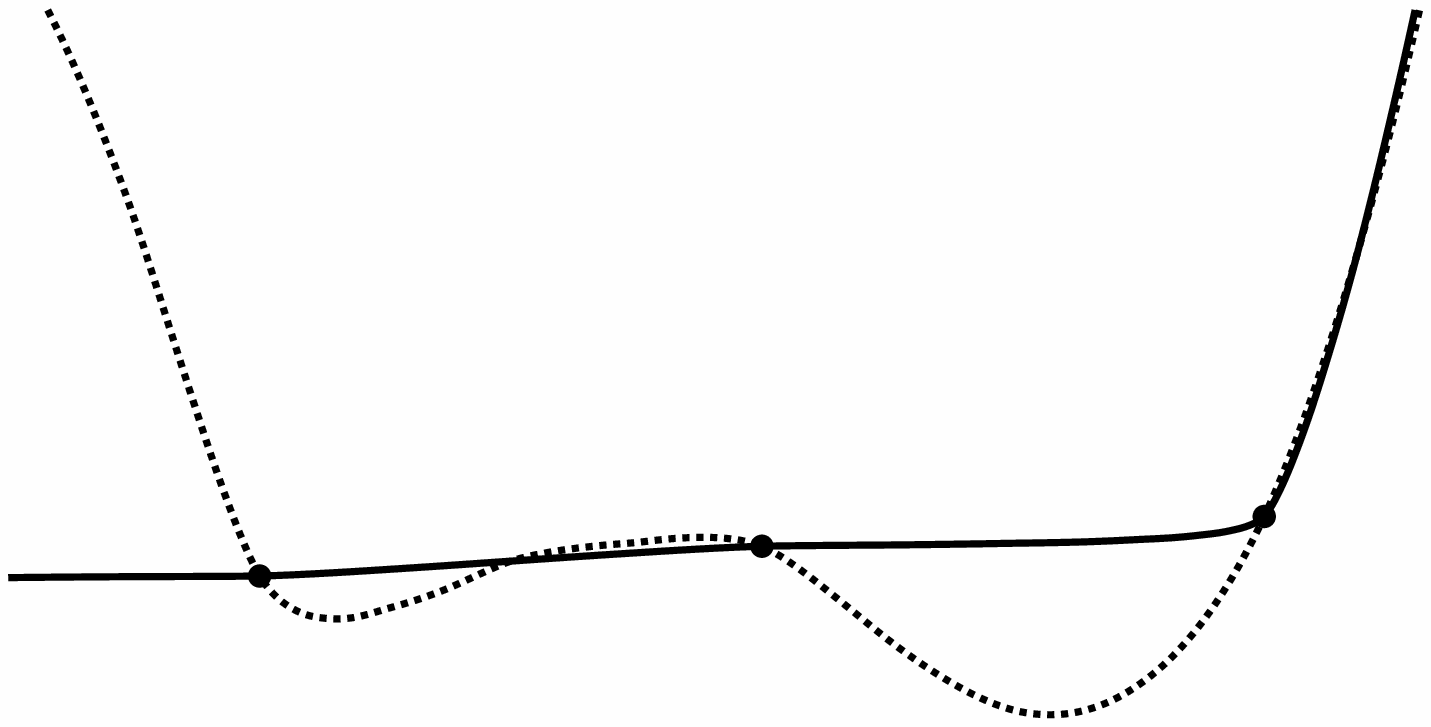}
\caption{Left plot: $T_{2,2}^\infty f^0$ (dotted line) and
  $S_\epsilon^\infty f^0$ (solid line) with  $f^0$ in
  \eqref{eq:monotone_data2}. Right plot: zoom of the area  marked with a dashed rectangle on the left plot.
\label{fig:monotone2}}
\end{figure}

In section \ref{sec:smoothness}, we have been able to prove that when
the data are strictly monotone (and appropriately chosen, see Theorem \ref{teo:smooth}),
the limit function is in fact $\cC^1$. In order to see the possible
differences between the limit functions for monotone and strictly
monotone data, we slightly modify the data in \eqref{eq:monotone_data1} to have the 
following initial set of strictly monotone data
\begin{equation} \label{eq:monotone_data2}
f^0=(10,10.1,10.2,10.3,10.4,10.5,10.6,10.7,10.8,15,50,50.1,50.2,50.3,60,85,85.1,85.2,85.3).
\end{equation}

In Figure \ref{fig:monotone2} we  display  the results corresponding
to $S^\infty_\epsilon f^0$ and $T_{2,2}^\infty f^0$, with the same
convention as in Figure \ref{fig:monotone}. Comparing the right plots
in Figures \ref{fig:monotone} and \ref{fig:monotone2}, it seems
evident that there is a difference in smoothness in both limit functions.
 To get a  numerical estimate of  the smoothness of the limit
 functions, we proceed 
as in \cite{DL-US17,Kuijt98}, 
and compute, in each case, 
\begin{equation*} 
\alpha \approx \log_2\left(\frac{\norma{\nabla^n f^{k}}}{\norma{\nabla^n f^{k+1}}}\right),
\end{equation*}
being $n$ a natural number greater than $\alpha$. In this case, $n\geq 2$.

{For the initial data in \eqref{eq:monotone_data1}, we find that
$\alpha\approx 1.34$. On
  the other hand, when $f^0$ is strictly monotone, as in
  \eqref{eq:monotone_data2}, we  get  $\alpha\approx 2$. We have also observed (numerically) this {\em improved}
  smoothness in other situations, for example in reconstructing the
  'heart' shown in Figure \ref{fig:heart}, where the numerical
  estimate gives $\alpha \approx 2$. The right display in this
  figure shows a smooth reconstruction of  initial 'edge' at the
  bottom.  Since
  $T_{2,2}$ is $\cC^{2-}$ we conjecture that this is, in fact, the
  smoothness of the limit functions shown in Figures
  \ref{fig:monotone2} and \ref{fig:heart}.

}

\begin{figure}
\centering
\includegraphics[width=0.3\textwidth, clip]{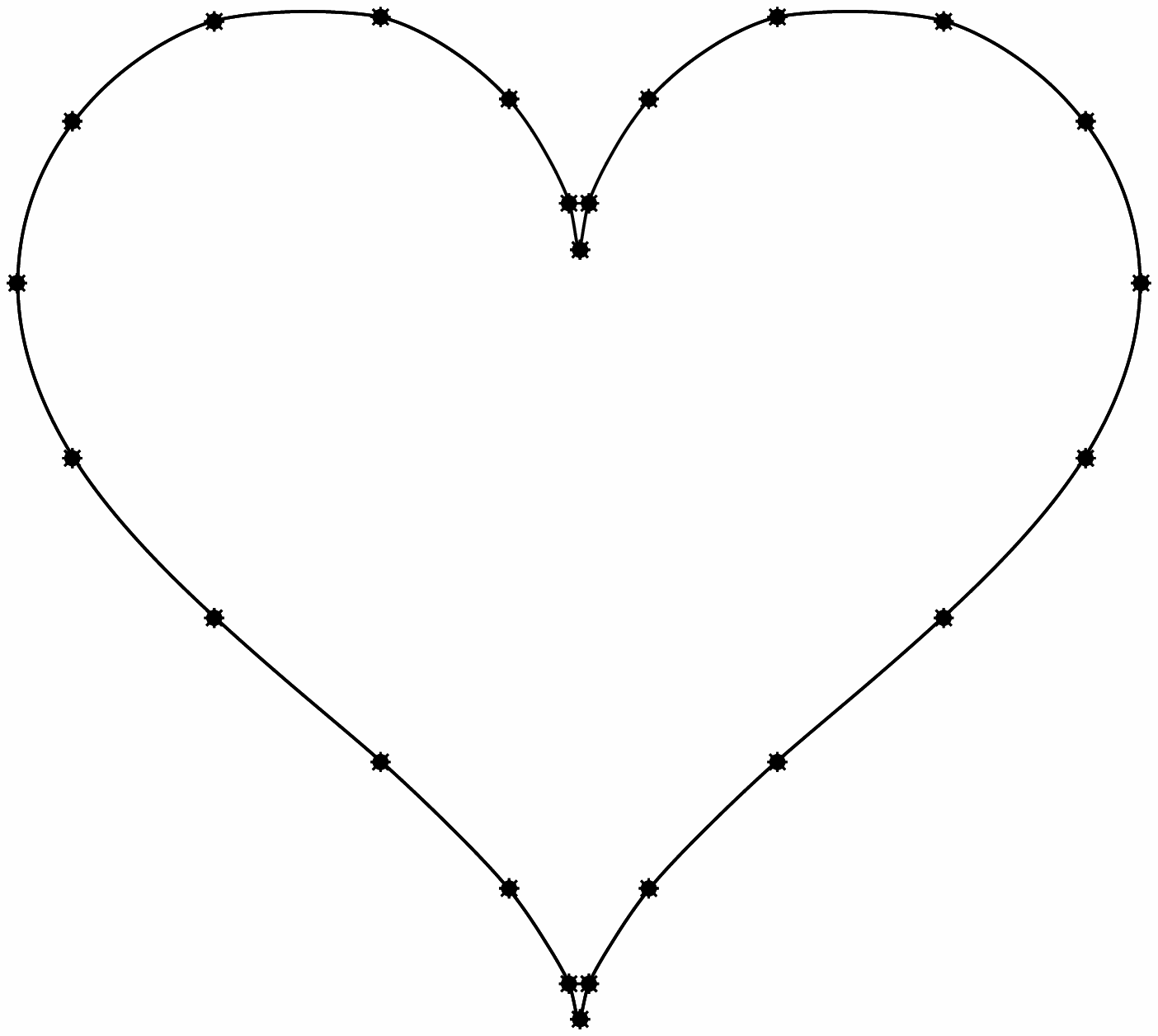}
\includegraphics[width=0.3\textwidth, clip]{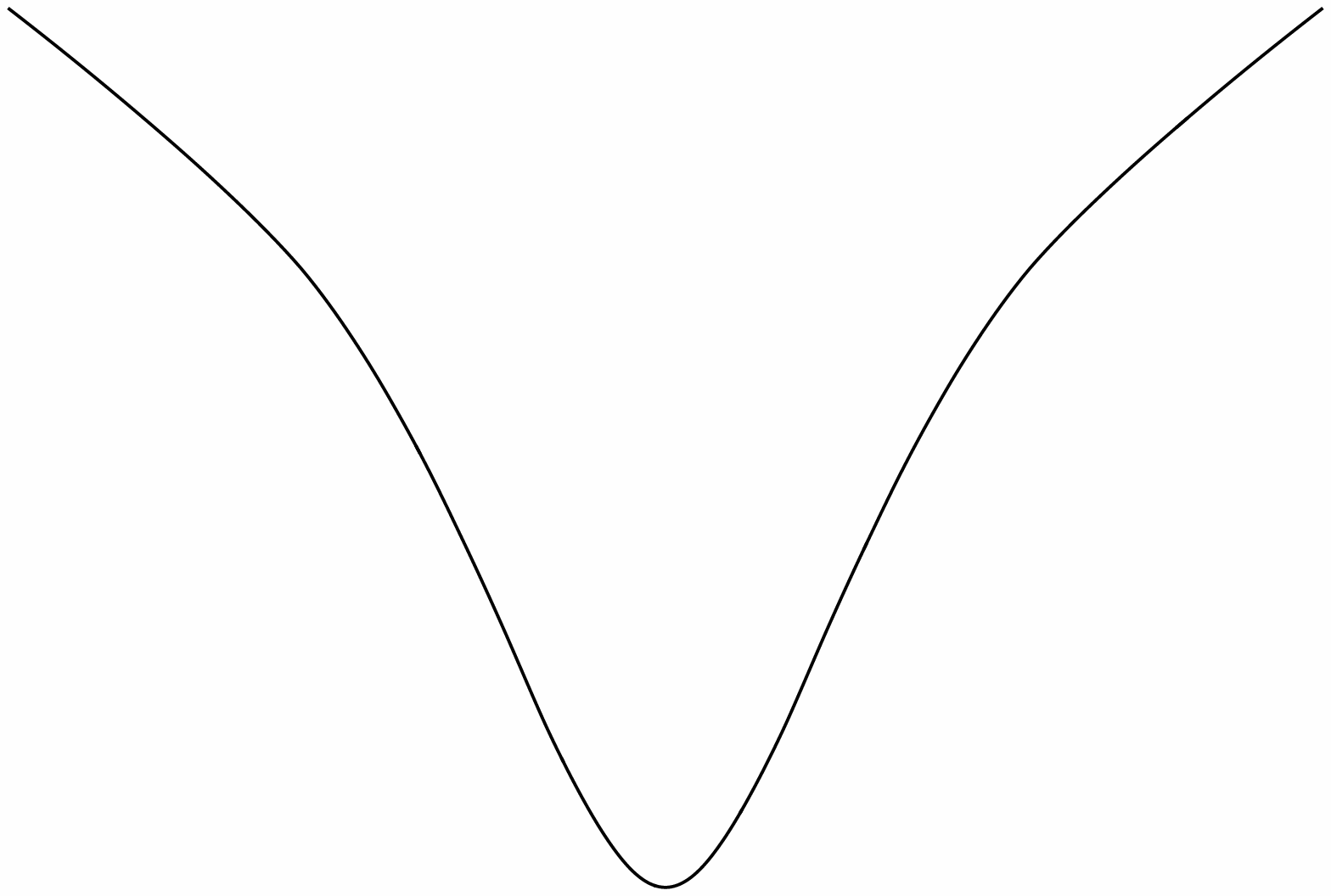}
\caption{Left: The curve generated (line) from the initial sequence (dots) using $S_\epsilon$. Right: A zoom of one edge of the left figure. 
\label{fig:heart}}
\end{figure}

Through extensive numerical testing, we have observed that the lowest
regularity numerically obtained for $S^\infty_\epsilon f^0$, for $f^0$ quite
arbitrary, is $\cC^\alpha$, with $\alpha \approx 1.34$. In
addition, this regularity seems to occur when considering
 four consecutive data  of the form
$$ f_{i-1}^k \leq f_i^k = f_{i+1}^k \leq f_{i+2}^k \text{ or } f_{i-1}^k \geq f_i^k = f_{i+1}^k \geq f_{i+2}^k,$$
which is exactly the case  in  Figure \ref{fig:monotone}. In such
situations $(S_\epsilon f^k)_{2i+1} = 
(T_{1,1} f^k)_{2i+1}$. This choice is motivated by the desire to ensure monotonicity, but it seems to
have an adverse effect on the smoothness of the limit function. In any
case, we also recall that, at this moment, only $\cC^1$ smoothness has
been proven (under appropriate restrictions).

\subsection{Approximation order}

We have proven in Theorem \ref{teo:approx}  that the approximation
order of $S_\epsilon$ after 
one iteration is 4, when $f=F|_{h \Z}$ and  $|F'|>\theta>0$. For such
data, we expect  that  stability results of section
\ref{sec:stable}, hence  we also expect that the 'strong'
approximation order is 4.

We check  numerically the approximation order obtained after refining
initial data sampled from the functions $F_1(t)=\exp(-2t^2)$, $F_2(t)=\exp(t)-t$, $t_i^k=i2^{-k}/100$.

Since the approximation order depends on
the monotonicity of $F$, as well as the properties of higher derivatives
of $F$,  we  measure the error, 
$$E_{k}:= \sup_{t_i^{k+7}\in[a,b]} \left| (S_\epsilon^7 (F_j(t_l^k))_{l\in\Z})_i - F_j(t_i^{k+7}) \right|, 	\qquad k\geq0,\, j=1,2.$$
at different  intervals $[a,b]$, after
seven applications of $S_\epsilon$. Notice that  when  $F'$ vanishes,
Theorem \ref{teo:approx} does not apply. 

\begin{table}
\centering
\begin{tabular}{|c|c|c|c|c|}\hline
  		& \multicolumn{2}{c|}{$F_1(t)=\exp(-2t^2)$}	& \multicolumn{2}{c|}{$F_2(t)=\exp(t)-t$}  	\\\hline\hline
$k$ & $E_{k}$ & $\log_2(E_{k}/E_{k-1})$ 		& $E_{k}$ &
$\log_2(E_{k}/E_{k-1})$ 	\\\hline \hline
0 &   5.5174e-09 & 				&   6.5725e-10 &			 \\\hline
1 &   3.4488e-10 &   3.9998e+00	&   4.1470e-11 &   3.9863e+00\\\hline
2 &   2.1555e-11 &   4.0000e+00	&   2.6044e-12 &   3.9931e+00\\\hline
3 &   1.3474e-12 &   3.9998e+00	&   1.6298e-13 &   3.9982e+00\\\hline \hline
0 &   3.4257e-09 & 				&   4.6993e-08 &			 \\\hline
1 &   2.1598e-10 &   3.9874e+00	&   5.8667e-09 &   3.0018e+00\\\hline
2 &   1.3557e-11 &   3.9938e+00	&   7.3288e-10 &   3.0009e+00\\\hline
3 &   8.4910e-13 &   3.9970e+00	&   9.1581e-11 &   3.0005e+00\\\hline
\end{tabular}
\caption{The error ($E_k$) and the approximation order
  ($\log_2(E_{k}/E_{k-1})$) of $S_\epsilon$ when approximating $F_1$
  (left) and $F_2$ (right) after 7 iterations. The top tabular region
  corresponds to the monotone region $[a,b]=[-1,-0.3]$. The bottom
  tabular region to the non-monotone region $[a,b]=[-0.4,0.4]$. \label{tab:approx}}
\end{table}
The results are summarized in Table \ref{tab:approx}.  Observe that
$|F_1'|,|F_2'|>\theta>0$ in $[-1,-0.3]$, thus the approximation order
is (shown in Table \ref{tab:approx} top) is 4, as expected. 
On the other hand, $F_1'(0)=0=F_2'(0)$, and the approximation order is
4 and 3, respectively, in the interval $[-0.4,0.4]$ (see Table
\ref{tab:approx} bottom).  To explain the difference, we may observe
the Taylor expansion in the proof of Theorem
\ref{teo:approx}. Different results may be expected depending on the
values of $F'(0),F''(0),F^{(3)}(0)$. Here, $0 = F_1'''(0) \neq
F_1''(0)$ and we obtained order 4,  while $F_2''(0)\neq 0 \neq
F_2^{(3)}(0)$ and we get order 3. 

which reproduces exactly third order polynomials

\section{Conclusions} \label{sec:conclusions}

In \cite{DLL03}, the authors derived non-stationary versions of the 
 classical four point Deslauriers-Dubuc linear scheme $T_{2,2}$. These
 linear schemes 
have the capability to reproduce  exactly the space of exponential
polynomials $\espan{1,t, \exp(\gamma t), \exp(-\gamma t)}$, but the level-dependent rules depend explicitly on the value
of the parameter $\gamma$ that defines the space. In practice,  this
value needs to be estimated  from the
initial data provided by the user. Hence, curves composed of different
conic sections are hard to reproduce using these linear schemes. 

In this paper, we have constructed a family of
nonlinear  schemes based on a nonlinear rule that can be 
considered as {\em stationary representative}  of  the linear, non-stationary
4-point schemes introduced in \cite{DLL03}.
 We 
 show that the schemes in this family reproduce exactly  second order
 polynomials. It also reproduces
 trigonometric  and hyperbolic functions, provided that some
 easily 
verifiable  conditions are fulfilled. We remark that no previous
knowledge on the parameters defining the hyperbolic/trigonometric
functions is required: the same scheme is being applied at all locations
of a curve
composed of different conic sections, obtaining exact reconstruction
away from the transition zones between sections.

We show that the new  schemes can be
written as a nonlinear perturbation of a linear scheme, as  in
\cite{ADL11,ADLT06,ADS16,DL08,DRS04,DL-US17,HO10}. The analysis of convergence,
monotonicity preservation and  
stability of the new schemes uses some of the tools developed in these
references. 

 We remark that the proximity theory \cite{Grohs08,Grohs10general},
 usually applied on manifold data subdivision schemes, cannot be
 applied in our case, because  our schemes do not verify a proximity condition.

In addition, we have shown that, for strictly monotone data and for a certain range
of the parameter that defines the cut-off function $\Gamma_{\epsilon}$, the nonlinear
rules become independent of the  value of this parameter, and are
smooth positive functions. This allowed us to prove that the
corresponding limit functions are $\cC^1$ smooth, provided that a (non
restrictive) technical condition is verified. Some numerical
experiments were carried out to support and validate 
the theoretical results obtained in the paper. 

The setting in this paper is one-dimensional. We   plan to extend
these ideas to define 
a new subdivision scheme, able to reproduce trigonometric functions in a
bivariate setting and  on triangular meshes. 


  

\section*{Acknowledgments}
\noindent
The authors acknowledge support from Project MTM2014-54388 (MINECO, Spain) and the FPU14/02216 grant (MECD, Spain).
We would like to thank the suggestion of Professor Ulrich Reif about the definition in \eqref{eq:GammaE-def}.

\bibliography{biblio}{}
\bibliographystyle{plain}

\begin{appendix}
\section{Gradient computations}

This appendix describes  the computation of the gradients that appear in section \ref{sec:smoothness}. We  recall that we are assuming that the data is strictly positive and $\epsilon \in [0, \sqrt{2}]$, hence  the subdivision rules of  $S_\epsilon^{(1)}$ are
\begin{align*}
\Psi_j (x,y,z) &= y +(-1)^j 2\Gamma^{[1]}_\epsilon(x,y,z) (x-z), \qquad j=0,1,
\end{align*}
where $\Gamma^{[1]}_\epsilon$, which does not depend on $\epsilon$, is given in the first row of \eqref{eq:Gamma1+}.

Let us denote $\uno_n = (1,1,\ldots,1)\in\R^{n}$. An easy computation leads to
$$ \nabla \Psi_j (\uno_3) = (0,1,0) + (-1)^j 2\Gamma^{[1]}(\uno_3)(1,0,-1)= (0,1,0) + (-1)^j \frac18 (1,0,-1).$$
Then
$$\nabla \Psi_0(\uno_3) = \left (\frac18,1,-\frac18\right ), 
\qquad \nabla \Psi_1(\uno_3) = \left (-\frac18,1,\frac18\right ).$$

Since $\Psi_0,\Psi_1:\R^3_+ \rightarrow \R$, 
 $G_1$ and $G_2$, in  \eqref{eq:G1},\eqref{eq:G2} are
positive and smooth too. Hence their gradients are well-defined. Some details are provided below, and the results are summarized  in Table \ref{tab:appendix:gradient1}. In the computations below, we use that $\Psi_j(\uno_3)=1$, the chain rule and the values obtained for $\nabla \Psi_j(\uno_3)$.

For $ G_1$ we have,
\begin{equation} \label{eq:G1}
G_1(x,y)= \frac{\Psi_1(x,1,y)}{\Psi_0(x,1,y)} \qquad \text{for } x,y>0,  
\qquad  \nabla (x,1,y) = \begin{pmatrix} 1& 0\\ 0& 0\\ 0& 1 \end{pmatrix},
\end{equation}
thus the chain rule leads to 
\begin{align*}
\nabla G_1(\uno_2) &= \frac{1}{\Psi_0(\uno_3)^2}\left (\Psi_0(\uno_3)\nabla \Psi_1(\uno_3)\begin{pmatrix} 1& 0\\ 0& 0\\ 0& 1 \end{pmatrix}- \Psi_1(\uno_3)\nabla \Psi_0(\uno_3)\begin{pmatrix} 1& 0\\ 0& 0\\ 0& 1 \end{pmatrix}\right ) \\
&= \left (\left (-\frac18,1,\frac18\right )- \left (\frac18,1,-\frac18\right )\right )\begin{pmatrix} 1& 0\\ 0& 0\\ 0& 1 \end{pmatrix}
= \left (-\frac14,\frac14\right ). 
\end{align*}
To compute $\nabla G_2(\uno_3)$, where
\begin{equation} \label{eq:G2}
 G_2(x,y,z) = \frac{\Psi_0(1,y,yz)}{\Psi_1(x,1,y)},
\end{equation}
we proceed analogously 
$$
 \nabla G_2(\uno_3)=\left (\frac18,1,-\frac18\right ) \begin{pmatrix} 0&0&0\\ 0&1&0\\ 0&1&1 \end{pmatrix} - \left (-\frac18,1,\frac18\right ) \begin{pmatrix} 1&0&0\\ 0&0&0\\ 0&1&0 \end{pmatrix}
 =\left (\frac18,\frac68,-\frac18\right ).
$$
\begin{table}[ht]
\centering
$$\arraycolsep=1pt\def\arraystretch{1.5} \begin{array}{|c|c|c|c|}\hline
 \, \|\Psi_0(\uno_3))\|_1  \, & \,  \|\Psi_1(\uno_3))\|_1  \, & \,  \|\nabla G_1(\uno_2)\|_1  \, & \,  \|\nabla G_2(\uno_3)\|_1\\\hline
5/4 & 5/4 & 1/2 & 1\\\hline
\end{array}$$
\caption{The 1-norms of the gradients of the subdivision rules $\Psi_0$, $\Psi_1$ and the functions $G_1$ and $G_2$. \label{tab:appendix:gradient1}}
\end{table}

To carry out the computations required in Lemma \ref{lem:contractivitat}, we use the following notation: 
The double application of $S_\epsilon$ is determined by
\begin{equation} \label{eq:Psij2}
\Psi^2_j(d_{i-2},d_{i-1},\ldots, d_{i+2}) := (S_\epsilon^{[1]}S_\epsilon^{[1]} d)_{4i+j}, \quad j=0,1,2,3,4,
\end{equation} where: 
\begin{align*}
(S_\epsilon^{[1]}S_\epsilon^{[1]} d)_{4i} &=\Psi_0\Big(\Psi_1(d_{i-2},d_{i-1},d_{i}),\Psi_0(d_{i-1},d_i,d_{i+1}),\Psi_1(d_{i-1},d_i,d_{i+1})\Big) ,\\
(S_\epsilon^{[1]}S_\epsilon^{[1]} d)_{4i+1} &=\Psi_1\Big(\Psi_1(d_{i-2},d_{i-1},d_{i}),\Psi_0(d_{i-1},d_i,d_{i+1}),\Psi_1(d_{i-1},d_i,d_{i+1})\Big) ,\\
(S_\epsilon^{[1]}S_\epsilon^{[1]} d)_{4i+2} &= \Psi_0\Big(\Psi_0(d_{i-1},d_i,d_{i+1}),\Psi_1(d_{i-1},d_i,d_{i+1}),\Psi_0(d_{i},d_{i+1},d_{i+2})\Big),\\
(S_\epsilon^{[1]}S_\epsilon^{[1]} d)_{4i+3} &= \Psi_1\Big(\Psi_0(d_{i-1},d_i,d_{i+1}),\Psi_1(d_{i-1},d_i,d_{i+1}),\Psi_0(d_{i},d_{i+1},d_{i+2})\Big).
\end{align*}
Applying the chain rule and the previous results, we get 
\begin{align*}
\nabla \Psi^2_0(\uno_5) &= 
(\frac18,1,-\frac18)\begin{pmatrix}
-\frac18 & 1 & \frac18 & 0 & 0\\
 0 & \frac18 & 1 & -\frac18 & 0\\
  0 & -\frac18 & 1 & \frac18 & 0 \end{pmatrix} = (-\frac{1}{64},\frac{17}{64},\frac{57}{64},-\frac{9}{64},0), \\
\nabla \Psi^2_1(\uno_5) &= (-\frac18,1,\frac18)\begin{pmatrix}
-\frac18 & 1 & \frac18 & 0 & 0\\
 0 & \frac18 & 1 & -\frac18 & 0\\
  0 & -\frac18 & 1 & \frac18 & 0 \end{pmatrix} = (\frac{1}{64},-\frac{1}{64},\frac{71}{64},-\frac{7}{64},0), \\
\nabla \Psi^2_2(\uno_5) &= (\frac18,1,-\frac18)\begin{pmatrix}
 0 & \frac18 & 1 & -\frac18 & 0\\
  0 & -\frac18 & 1 & \frac18 & 0\\
0 & 0 & \frac18 & 1 & -\frac18 \end{pmatrix} = (0,-\frac{7}{64},\frac{71}{64},-\frac{1}{64},\frac{1}{64}), \\
\nabla \Psi^2_3(\uno_5) &= (-\frac18,1,\frac18)\begin{pmatrix}
 0 & \frac18 & 1 & -\frac18 & 0\\
  0 & -\frac18 & 1 & \frac18 & 0\\
0 & 0 & \frac18 & 1 & -\frac18 \end{pmatrix}  = (0,-\frac{9}{64},\frac{57}{64},\frac{17}{64},-\frac{1}{64}), \\
\nabla \Psi^2_4(\uno_5) &= (\frac18,1,-\frac18)\begin{pmatrix}
0 & -\frac18 & 1 & \frac18 & 0 \\
0 & 0 & \frac18 & 1 & -\frac18 \\
0 & 0 & -\frac18 & 1 & \frac18 \end{pmatrix} = (0,-\frac{1}{64},\frac{17}{64},\frac{57}{64},-\frac{9}{64}).
\end{align*}
Then, for $j=0,1,2,3$, 
\begin{equation} \label{eq:Gsup2}
G_j^2(x,y,z,w):=\frac{\Psi^2_{j+1}(x y,y,1,z, z w)}{\Psi^2_j(x y,y,1,z , z w)}, 
\qquad  
\nabla G_j^2(\uno_4) = (\nabla \Psi^2_{j+1}(\uno_5) -\nabla \Psi^2_{j}(\uno_5) )
\begin{pmatrix}
1&1&0&0\\
0&1&0&0\\
0 & 0& 0& 0\\
0&0&1&0\\
0&0&1&1\\
\end{pmatrix} 
\end{equation}
so that
\begin{align*}
\nabla G^2_{0} (\uno_4) &= (\frac{1}{32},-\frac14,\frac{1}{32},0), & \nabla G^2_{2}(\uno_4) &= (0,-\frac{1}{32},\frac14,-\frac{1}{32}),\\
\nabla G^2_{1}(\uno_4)&= (-\frac{1}{64},-\frac{7}{64},\frac{7}{64},\frac{1}{64}) , & \nabla G^2_{3}(\uno_4) &= (0,\frac18,-\frac12,-\frac18).
\end{align*}
The relevant results are summarized in  Table \ref{tab:appendix:gradient2}.
\begin{table}[!h]
\centering
\begin{tabular}{|c|c|c|c|c|} \hline
&&&&\\[-12pt]
$j$ & 0 & 1 & 2 & 3 \\\hline
&&&&\\[-10pt]
$\|\nabla G^2_{j}(\uno_4)\|_1$ &  \, 5/16 \, & \, 1/4 \, & \, 5/16 \, & \, 3/4 \, \\[2pt]\hline
\end{tabular}
\caption{The 1-norm of the ratio functions $G^2_j$ for $j=0,1,2,3$. \label{tab:appendix:gradient2}}
\end{table}
\end{appendix}

\end{document}